\documentclass[12pt,reqno]{amsart}

\usepackage{amsmath, amscd, amssymb, amsthm,amsfonts}
\usepackage{euscript, mathrsfs, latexsym,mathabx,stmaryrd}
\usepackage{color}

\usepackage{lmodern}
\usepackage[T1]{fontenc} 

\usepackage{multicol}

\ifx\pdfoutput\undefined
\usepackage{graphicx}
\else
\fi

\ifx\xetex
\usepackage{fontspec}
\usefont{EU1}{lmr}{m}{n}
\else
\fi

\usepackage{url} 
\usepackage{tikz}
\usetikzlibrary{matrix,arrows,positioning}
\tikzset{node distance=2cm, auto}

\usepackage[nodayofweek]{datetime}

\parskip=\medskipamount

\newcommand{\conj}[1]{\quad\textnormal{ #1 }\quad}

\newcommand{\inp}[1]{\ensuremath{\langle #1 \rangle}}

\newcommand{\normaltext}[1]{\textnormal{#1}}

\makeatletter
\def\imod#1{\allowbreak\mkern2.5mu({\operator@font mod}\,#1)}
\makeatother

\renewcommand{\a}{\alpha}

\newcommand{\opp}{\oplus}
\newcommand{\ott}{\otimes}

\renewcommand{\l}{\lambda}

\newcommand{\ga}{\gamma}

\renewcommand{\d}{\delta}

\newcommand{\aA}{\mathcal{A}}

\newcommand{\aL}{\mathcal{L}}
\newcommand{\aM}{\mathcal{M}}

\newcommand{\aT}{\mathcal{T}}

\newcommand{\bA}{\mathbf{A}}

\newcommand{\fA}{\mathfrak{A}}

\newcommand{\fD}{\mathfrak{D}}

\usepackage{bbm}

\newcommand{\CC}{\mathbb{C}}

\newcommand{\FF}{\mathbb{F}}

\newcommand{\ZZ}{\mathbb{Z}}

\newcommand{\kk}{\mathbb{k}}

\newcommand{\eA}{\EuScript{A}}
\newcommand{\eB}{\EuScript{B}}
\newcommand{\eC}{\EuScript{C}}
\newcommand{\eD}{\EuScript{D}}

\theoremstyle{plain}

\newtheorem{thm}{Theorem}[section]

\newtheorem{theorem}[thm]{Theorem}

\newtheorem{conjecture}[thm]{Conjecture}

\newtheorem{proposition}[thm]{Proposition}
\newtheorem{prop}[thm]{Proposition}

\newtheorem{corollary}[thm]{Corollary}
\newtheorem{cor}[thm]{Corollary}
\newtheorem{lemma}[thm]{Lemma}

\theoremstyle{remark}

\theoremstyle{definition}

\newtheorem{example}[thm]{Example}

\newtheorem{defn}[thm]{Definition}

\newtheorem{definition}[thm]{Definition}

\newtheorem{notation}[thm]{Notation}

\newtheorem{remark}[thm]{Remark}

\numberwithin{equation}{section}

\pdfoutput=1

\usepackage{mathtools} 
\usetikzlibrary{knots}
\usetikzlibrary{patterns,decorations.pathreplacing}
\usepackage[colorlinks,linkcolor=blue,citecolor=blue]{hyperref}
\usepackage{overpic}
\usepackage{enumitem}   
\usepackage{graphicx,wrapfig}
\usepackage{caption}
\usepackage{pinlabel} 

 \usepackage{dutchcal}

\begin{document}

\title{Holonomicity from a Heegaard-Floer Perspective}
\date{\today}
\author[Benjamin Cooper]{Benjamin Cooper}
\author[Robert Deyeso III]{Robert Deyeso III}
\address{University of Iowa, Department of Mathematics, 14 MacLean Hall, Iowa City, IA 52242-1419 USA}
\email{ben-cooper\char 64 uiowa.edu}
\email{robert-deyeso\char 64 uiowa.edu}

\def\JS#1{\textcolor[rgb]{0,.75,.8}{ [JS: #1]}}
\def\BC#1{\textcolor[rgb]{0,.5,1}{ [BC: #1]}}
\def\OLD#1{\textcolor[rgb]{1,.69,.4}{ [OLD STUFF: #1]}}
\def\OPT#1{\textcolor[rgb]{1,.61,0}{ [OPTIONAL: #1]}}
\def\RD#1{\textcolor[rgb]{1,0.5,0}{RD: #1}}

\newcommand{\defeq}{\vcentcolon=}

\newcommand{\tm}{\widetilde{m}}
\newcommand{\dga}[0]{\operatorname{-dga}}
\newcommand{\dgca}[0]{\operatorname{-dgcoa}}
\newcommand{\Om}{\Omega}
\newcommand{\pa}{\partial}

\newcommand{\Ainf}{A_\infty}
\newcommand{\varB}[1]{{\operatorname{\mathit{#1}}}}
\renewcommand{\slash}[1]{H_/(#1)}
\newcommand{\pAinf}[0]{\varB{p-\Ainf}\!}
\newcommand{\semp}{{\{\emptyset\}}}
\newcommand{\z}{z}
\newcommand{\T}{T}
\renewcommand{\d}{\delta}
\newcommand{\h}{h}
\renewcommand{\t}{t}
\newcommand{\G}{\Gamma} 
\newcommand{\A}{\Lambda} 
\renewcommand{\bA}{\bar{\Lambda}} 

\renewcommand{\kk}{k}

\newcommand{\lab}[1]{\normaltext{#1}}
\newcommand{\nt}[1]{\normaltext{#1}}
\newcommand{\Ya}{\mathcal{Y}\hspace{-.1475em}a_{2,2}}
\newcommand{\vnp}[1]{\lvert #1 \rvert}
\newcommand{\vvnp}[1]{\lvert\lvert #1 \rvert\rvert}
\newcommand{\I}{[0,1]}
\newcommand{\xto}[1]{\xrightarrow{#1}}
\newcommand{\xfrom}[1]{\xleftarrow{#1}}
\newcommand{\from}{\leftarrow}

\newcommand{\dgcat}{\normaltext{dgcat}_k}
\newcommand{\Forget}{\normaltext{Forget}}
\newcommand{\Mat}{Mat}
\newcommand{\op}{\normaltext{op}}
\newcommand{\ab}{\normaltext{ab}}

\newcommand{\Vect}{Vect}
\newcommand{\Kom}{Kom}
\newcommand{\Set}{Set}
\newcommand{\Ch}{Ch}
\newcommand{\Tw}{Tw}
\newcommand{\Ob}{Ob}
\newcommand{\Hom}{Hom}
\renewcommand{\d}{\delta}

\newcommand{\pre}[0]{\operatorname{pre-}}
\newcommand{\coker}[0]{\operatorname{coker}}
\newcommand{\im}[0]{\operatorname{im}}

\newcommand{\KB}{\mathcal{K}_q}
\newcommand{\QT}{A_q}
\newcommand{\Zt}{{\ZZ/2}}
\newcommand{\Basep}{\mathcal{B}'}
\newcommand{\Base}{\mathcal{B}}

\newcommand{\cfd}{\widehat{\textit{CFD}}}
\newcommand{\cfk}{\widehat{\textit{CFK}}}

\newcommand{\scfk}{S^r\widehat{\textit{CFK}}}
\newcommand{\sscfk}{S^{r+1}\widehat{\textit{CFK}}}
\newcommand{\gscfk}{S^r\widehat{\textit{CFK}}}
\newcommand{\shfk}{S^r\widehat{\textit{HFK}}}
\newcommand{\hfk}{\widehat{\textit{HFK}}}
\newcommand{\gr}{gr}
\newcommand{\fcfk}{S^r\widehat{\textit{CFK}}}
\newcommand{\ffcfk}{S^{r+1}\widehat{\textit{CFK}}}
\newcommand{\hf}{\widehat{\textit{HF}}}
\newcommand{\gvsp}{\Vect_\ZZ}
\newcommand{\chs}{\Ch^*(\gvsp)}
\newcommand{\lnp}[1]{\ensuremath{[\![#1]\!]}}
\newcommand{\pnp}[1]{\ensuremath{P(#1)}}

\newcommand{\alex}{\Delta}
\newcommand{\blex}{\alex}
\newcommand{\unalex}{\widebar{\alex}}
\newcommand{\sralex}{\alex^r}
\newcommand{\srmalex}{\alex^{r-1}}
\newcommand{\unsralex}{\widebar{\alex}^r}
\newcommand{\unsrmalex}{\widebar{\alex}^{r-1}}

\newcommand{\unsonealex}{\widebar{\alex}^1}
\newcommand{\sonealex}{\alex^1}

\newcommand{\hT}{\ZZ[[q]]}
\newcommand{\spinc}{\text{spin}^{c}}
\newcommand{\tws}{\aT}
\newcommand{\Cone}{C}
\newcommand{\ul}[1]{\overline{#1}}

\begin{abstract}
We construct $S^r$-colored knot Floer homologies and prove that they satisfy
categorified recurrence relations. The associated Euler characteristic
implies $q$-holonomicity of the corresponding sequence of colored Alexander
polynomials, inspired by the AJ conjecture for colored Jones polynomials.
  \end{abstract}

\maketitle
\setcounter{tocdepth}{1}
\setcounter{secnumdepth}{2}
\tableofcontents

\section{Introduction}\label{introsec}
\newcommand{\tS}{S}
\newcommand{\tR}{\ZZ[q,q^{-1}]}
\newcommand{\hR}{\tR}
\newcommand{\ZZg}{\ZZ_{\geq 0}}

Given a knot $K$ in $S^3$ and a simple Lie algebra $\mathfrak{g}$, the
Reshetikhin-Turaev construction produces a $q$-series valued invariant
$\inp{K}_V$ of $K$ for each finite dimensional representation $V$ of
the corresponding quantum group $U_q(\mathfrak{g})$.  A priori, the collection 
$\{\inp{K}_V \}_{V\in Rep(\mathfrak{g})}$ appears to contain an infinite amount of information, but the existence of
recurrence relations among, or the $q$-holonomicity of, these knot invariants
tells us that there is a great deal of redundancy. In exchange for an infinite
sequence
$$f : \ZZg \to \ZZ[q^{-1}][[q]]$$
we can make do with a finite collection of initial conditions and
a recurrence relation of the form
\begin{gather}\label{qholeqn}
\sum_{i=n}^{n+m} a_i(q^n,q) f(i) =0 \conj{ for } a_i(u,v)\in R[u,v],
\end{gather}
where $R := \ZZ[q,q^{-1}]$. Garoufalidis and Le proved that for each knot $K$ such a recurrence relation
always exists among $\mathfrak{g}$ invariants (for $\mathfrak{g} \neq G_2$)
\cite[Thm. 6]{GL}.  Unfortunately, this approach does not lead to an
interpretation of the recurrence relation.  In the simplest non-trivial
case, the irreducible representations $V_n \cong S^n(V_1)$ of
$\mathfrak{sl}_2$ ($\dim V_n = n+1$) determine a sequence
$$f(n) = \inp{K}_{V_n}$$
called the colored Jones polynomials of $K$. The AJ conjecture, posed by Garoufalidis \cite{G1,G2}, states that the recurrence relation satisfied by this sequence is a non-commutative deformation of the $A$-polynomial. The commutative $A$-polynomial is characterized by vanishing on the space of characters which extend from the boundary torus $S^1\times S^1 \cong \partial (S^3\backslash \nu K)$ into the $3$-manifold $S^3\backslash \nu K$, where $\nu K$ is a regular neighborhood of the knot $K$, see \cite{CCGLS}.
In foundational work, Frohman, Gelca and Lofaro \cite{FGL,FGSKM} studied similar statements. 
Gukov studied this problem using insights from physics \cite{Gukov}. 
The conjecture is known to be true for some families of knots, see \cite{Le}.

A different story which also stems from the Reshetikhin-Turaev invariants
involves Khovanov homology and its generalizations, knot homology theories, which categorify the Reshetikhin-Turaev invariants, see \cite{KR1,Webster}. For a knot $K$ in $S^3$, and a representation $V$ of a simple Lie algebra $\mathfrak{g}$,  there are chain complexes of graded vector spaces $\lnp{K}_V$, the homology of which are invariants of $K$. The generating function associated to the homology is the Poincar\'{e} polynomial
\begin{equation}\label{decateq}
  \pnp{K}_V(t,q) = \sum_{i,j\in \ZZ} t^i q^j \dim H^{i,j}(\lnp{K}_V), \conj{ satisfying } \pnp{K}_V|_{t=-1} = \inp{K}_V.
  \end{equation}
Its value at $t=-1$ recovers the 
Reshitikhin-Turaev invariants as the graded Euler characteristic of knot homology. 
The situation here is sometimes likened to the relationship between ordinary homology and Euler characteristic, while both are useful tools, the historical precedence of the latter has been overshadowed by the great worlds which have been revealed to us by the former.

If one considers super Lie algebras $\mathfrak{g} = \mathfrak{gl}(1|1)$ then the
Alexander polynomial $\alex_K(q)$ is an instance of the
Reshitikhin-Turaev construction outlined above \cite{Viro,Sartori}.
Heegaard-Floer homology was used to define knot Floer homology $\hfk(K)$ by
Ozsv\'{a}th, Szab\'{o} and Rasmussen \cite{OZ,R}. In analogy with
Eqn. \eqref{decateq}, the Euler characteristic of this knot homology
theory is the Alexander polynomial.  

Our interest in this paper is to examine the question of $q$-holonomicity for these knot homology theories.  
The HOMFLY homology unites the
$\mathfrak{sl}_n$ knot homology theories \cite{RKR,cautisremarks} and,
conjecturally, this includes the knot Floer homology as well
\cite{DGR,BPLW}.
Fuji, Gukov and Su\l kowski \cite{FGS} found examples of recurrence relations
among the Poincar\'{e} polynomials $\pnp{K}_{S^n}(q,t)$ of the HOMFLY homology theories. The colored HOMFLY polynomials $\pnp{K}_{S^n}(q,-1)$ are known to be $q$-holonomic \cite{GLL}.
When placed within the framework of \cite{DGR}, this work suggests the existence of some categorified recurrence relations or {\em homological $q$-holonomicity} among $S^n$-colored HOMFLY knot homologies; a lifting of those relations.
It is therefore natural to formulate the conjecture below.

\begin{conjecture}\label{conjconj}
  There exists a notion of homological $q$-holonomicity for sequences of chain complexes of (filtered) graded vector spaces.
For each knot $K$ in $S^3$, the sequence of $S^r$-colored HOMFLY homologies 
  $\{\lnp{K}_{S^r}\}_{r\geq 1}$
is homologically $q$-holonomic.
\end{conjecture}

Here we pursue a simplified version of this conjecture. Since a spectral sequence from HOMFLY homology to knot Floer homology ought to imply that homological $q$-holonomicity holds for knot Floer homology as well. We restrict our attention to this setting, where a
robust theory has been developed and can be leveraged in the context of our problem.

\begin{theorem}
  There exists a notion of homological $q$-holonomicity for sequences of (filtered) chain complexes of graded vector spaces. For each knot $K$ in $S^3$, the $S^r$-colored knot Floer homologies $\{\fcfk(K)\}_{n\geq 1}$ is homologically $q$-holonomic.
\end{theorem}

{\bf Organization.} This paper is structured as follows. In \S\ref{classicalsec} we show
that the $S^r$-colored Alexander polynomials are $q$-holonomic in the
sense of Eqn. \eqref{qholeqn}.
In \S\ref{altsec}, using immersed curve technology
\cite{HRW2}, we construct the $S^r$-colored knot Floer homologies
$\shfk(K)$ using Rozansky's infinite braid argument \cite{roz} (which the spectral
sequence from the $S^r$-colored HOMFLY homology \cite[Thm. 2.3]{hogancamp} would
naturally abut). The graded Euler characteristics of $\shfk(K)$ correspond to limits
of Alexander polynomials $\sralex_K(q)$ from \S\ref{classicalsec}. 
We introduce a notion of homological $q$-holonomicity in \S\ref{wahhsec} which would apply to Conjecture \ref{conjconj} above and discuss a few simple examples.  In \S\ref{recsec} we prove that the $S^r$-colored knot Floer
homologies are homologically $q$-holonomic.

{\bf Conventions.} All of the diagrams of immersed curves have been rotated by $-\pi/2$.  When discussing chain complexes and homology, coefficients are always taken to be in the field $\FF_2$.

{\bf Acknowledgements.} The first author would like to thank M. Sto\v{s}i\'c and O. Yacobi for conversations. The authors thank J. Hanselman for his correspondence.

After distributing the preprint of this paper, it was brought to our
attention that D. Chen studied stablization of the Alexander polynomial for
links with full twists \cite{Chen22} and P. Lambert-Cole studied stability
for knot Floer homology under the operation of replacing a crossing with
$n$-twists \cite{Lambertcole}. It would be interesting to compare their work
to materials in Sections \ref{classicalsec} and \ref{altsec} respectively.

\section{Holonomicity of colored Alexander polynomials}\label{classicalsec}
In \S \ref{ncoloredalexsec} we construct the $S^n$-colored Alexander polynomials and in \S \ref{qholalexsec} we show that this sequence is $q$-holonomic.

\subsection{$S^n$-colored Alexander polynomials}\label{ncoloredalexsec}
For a knot $K$ in $S^3$, the Alexander polynomial $\alex_K(q)\in\ZZ[q,q^{-1}]$ \cite{Lickorish}. A priori, $\alex_K(q)$ is only defined up to multiplication by $\pm q^{\pm 1}$, but it can be expressed uniquely when the equation
$$\alex_K(q) = \alex_K(q^{-1})$$
is required to hold (and the lowest order term is positive). In this paper, $\alex_K(q)$ will be used to denote this {\em symmetrized form} of the Alexander polynomial. 

A different normalization, $\alex^1_K(q)$ which will be more important for us in what follows, is determined by requiring the lowest power of $q$ to be $q^0$. More precisely, if $\deg(\alex_K)$ is the largest mononomial power of $q$ occurring in the symmetrized form of $\alex_K(q)$ then the polynomial
\begin{equation}\label{shiftedalexeq}
  \alex^1_K(q) := q^{\deg(\alex_K)} \alex_K(q)\in \ZZ[q]
  \end{equation}
is the {\em positive form}. The $S^r$-colored Alexander polynomials are given by certain limits of positive forms of Alexander polynomials described in this section.

\begin{defn}\label{def:cable}
Given a knot $K$ in $S^3$ and $\nu K \subset S^3$ a regular neighborhood. The {\em $(r,s)$-cable} $K_{r,s}$ of $K$ is the $(r,s)$-torus link in the boundary torus $\partial \nu K$, winding $r$-times around the meridian and $s$-times around the longitude that given by the Seifert framing. When $r$ and $s$ are relatively prime, $\gcd(r,s)=1$, the link $K_{r,s}$ is a knot.
\end{defn}

The Alexander polynomial $\alex_{K_{r,s}}(q)$ of the $(r,s)$-cable $K_{r,s}$ can be expressed in terms of the Alexander polynomial $\alex_K(q)$ of the knot $K$ and the Alexander polynomial of the torus link $\alex_{T_{r,s}}(q)$ respectively \cite{Lickorish}. 
\begin{align}
  \alex_{K_{r,s}}(q) &= \alex_K(q^r) \alex_{T_{r,s}}(q) \nonumber\\
  &= \alex_K(q^r)\cdot \frac{q^{rs}-1}{q^r-1}\cdot \frac{q-1}{q^s-1} \label{cableeq}
\end{align}
By writing the rational functions as geometric series as in Eqn. \eqref{geometricserieseq}, 
$\alex^1_{K_{r,s}}(q) = q^{\deg(\alex_{K_{r,s}})} \alex_{K_{r,s}}(q)$ 
becomes an element of the power series ring
\begin{equation}
\ZZ[[q]] := \varprojlim \ZZ[q]/(q^n).
\end{equation}

The proposition below shows that the sequence of Alexander polynomials obtained by increasing the number of longitudinal full twists converges in the topology of the ring $\ZZ[[q]]$.

\begin{prop}\label{limcableprop}
  There is a limit
  $$\lim_{n\to \infty} \alex^1_{K_{r,rn+1}}(q) = \alex^1_K(q^r) \frac{q-1}{q^r-1},$$
where the rational function represents the power series
\begin{equation}\label{geometricserieseq}
  \frac{q-1}{q^r-1} := (1-q)\sum_{i=0}^\infty q^{ir}.
  \end{equation}
  \end{prop}
\begin{proof}
For a sequence $\{ f_n \}_{n\geq 0}$ of power series to converge to $f$ in the $q$-adic topology, or
$f_n \to f$ as $n\to\infty$, it suffices to show that for each $n\in \ZZ_{\geq 0}$, there exists $N\in \ZZ_{\geq 0}$ such that 
\begin{equation*}
f\equiv f_i\!\!\!\pmod{q^n} \conj{ for all } i > N.
\end{equation*}

When $s>r$, the Alexander polynomial $\alex_{T_{r,s}}(q)$ satisfies
\begin{equation}\label{rateq}
\frac{q-1}{q^r-1}  \frac{q^{rs}-1}{q^s-1} = \frac{q-1}{q^r-1}\cdot 1 + \frac{q-1}{q^r-1}\cdot q^s + \cdots + \frac{q-1}{q^r-1}\cdot q^{(r-1)s}
  \end{equation}
Using Eqn. \eqref{cableeq} and \eqref{rateq}, it follows that 
 \begin{align*}
 \lim_{n\to\infty} q^{\deg(\alex_{K_{r,rn+1}})} \alex_{K_{r,rn+1}}(q) &= q^{r\deg(\alex_K)} \alex_K(q^r) \lim_{n\to \infty} \left( \frac{q-1}{q^r-1} + q^{rn+1}(\ast)\right)\\
 &= q^{r\deg(\alex_K)} \alex_K(q^r) \left( \frac{q-1}{q^r-1} + 0\right)
 \end{align*}
where $(\ast)$ corresponds to the higher order terms on the right-hand side of in Eqn. \eqref{rateq}.
\end{proof}

\begin{definition}\label{coloredalexdef}
The {\em $S^r$-colored Alexander polynomial} $\unsralex_K(q)$ of a knot $K$ is the limit of Alexander polynomials of the $(r,rn+1)$-cables of $K$ as $n\to \infty$, or
$$\unsralex_K(q) := q^{r\deg(\alex_K)} \alex_K(q^r) \frac{q-1}{q^r-1}$$
When $K$ is the unknot $U$, the Alexander polynomial is $1$; $\alex_U(q) = 1$. So, the $S^r$-colored Alexander polynomial of the unknot is the series below.
\begin{equation}\label{unknoteq}
\unsralex_U(q) = \frac{q-1}{q^r-1} = \sum_{i=0}^\infty (q^{ir} - q^{ir+1})
\end{equation}
The {\em reduced $S^r$-colored Alexander polynomial} of a knot $K$ is the quotient 
\begin{equation}\label{redeq}
\sralex_K(q) := \unsralex_K(q)/\unsralex_U(q).
\end{equation}  
Notice that $\sonealex_K(q)$ agrees with $\unsonealex_K(q)$ and Eqn. \eqref{shiftedalexeq} above.

\end{definition}  

\subsection{Recurrence relation for $S^r$-colored Alexander polynomials}\label{qholalexsec}
\newcommand{\tA}{A}
\newcommand{\weyl}{\aA}
\newcommand{\Do}{D}
\newcommand{\Seq}{\textnormal{Seq}}
\newcommand{\halfbang}{\kern-1.6em}
\newcommand{\SSeq}{\mathcal{Seq}}

In the remainder of this section we will show directly that the sequence of
$r$-colored Alexander polynomials satisfy a recurrence relation of the
form Eqn. \eqref{qholeqn} from the introduction. The relative simplicity of the
expressions $\sralex_K(q)$ and $\unsralex_K(q)$ from Def. \ref{coloredalexdef} allow us to be explicit. This construction will be 
refined in \S \ref{recsec}.

\begin{defn}{(\Seq)}\label{seq}
Let $R:=\ZZ[q,q^{-1}]$. If $T$ is an $R$-algebra then there is a natural $T$-module 
$\Seq(T) := T^{\ZZg}$ which consists of sequences of elements of $T$ which are parameterized by $r\in\ZZg$. Let $\Seq := \Seq(R[[q]])$ be sequences of power series
which are bounded in negative powers of $q$. 
\end{defn}

For each knot $K$, using Def. \ref{coloredalexdef}, there are sequences
\begin{gather}\label{alexseq}
\alex_K,\unalex_K : \ZZg \to R[[q]]\in \Seq \conj{ given by }\\ \alex_K(r) := \sralex_K(q)\conj{ and } \unalex_K(r) := \unsralex_K(q)
\end{gather}
of reduced and unreduced $S^r$-colored Alexander polynomials respectively.
We will find recurrence relations of the form
\begin{gather}\label{receqn}
\sum_{i=n}^{n+m} a_i(q^n,q) \alex_K(i) =0 \conj{ for } a_i(u,v)\in R[u,v]
\end{gather}
This will be done in two steps. First in \S\ref{Dopsec} we find such a recurrence relation
 $\tA_K$ for the sequence of reduced polynomials $\alex_K$.
Then in \S\ref{unknotsec} this relation is combined with the recurrence relation Eq. \eqref{unknoteq} for the unknot $\alex_U(q)$ to produce a recurrence relation for the unreduced sequence $\unalex_K$. 

Before we start, a bit of material about the relationship between recurrence relations and the Weyl algebra is reviewed.

\subsubsection{The Weyl algebra}\label{weylalgsec}
It is sometimes useful to think of recurrence relations in terms of operators.
There are two operators $L,M$ acting on sequences $f \in \Seq$ by
\begin{equation}\label{lmeq}
  (Lf)(n) := f(n+1)\conj{ and } (Mf)(n) := q^n f(n)
  \end{equation}
These operators satisfy the relation $q LM = ML$. The {\em Weyl algebra} is defined by this data
$$\weyl := \ZZ[q,q^{-1}]\inp{L,M}/\inp{qLM - ML}$$
and by virtue of the assignments in Eqn. \eqref{lmeq},  $\Seq$ is an $\weyl$-module.
For a given sequence $f\in \Seq$, the recurrence relation in Eqn. \eqref{receqn} is equivalent to the relation
$$A f = 0$$
when $A\in \weyl$ is the element $A = \sum_{i=n}^{n+m} L^ia_i(M,q)$, $a_i\in\ZZ[q,q^{-1}]\inp{M}$ for $i=n,\ldots,n+m$. The set of recurrence relations satisfied by $f$ can be identified with
$$Ann(f)=\{A\in\weyl : Af = 0\}\subseteq \weyl.$$

\subsubsection{$\Do$-operators}\label{Dopsec}
Let $X\in \ZZ[q,q^{-1}]$ be a $q$-Laurent polynomial. Then there is an operator  $\Do_X= 1-XL\in \weyl$ which acts on sequences. In long form, if $f \in \Seq$ is a sequence then
$$(\Do_Xf)(n) := f(n) - X f(n+1)$$
is a new sequence which is defined in terms of $f$ and multiplication by $X$.

\begin{prop}\label{dpropprop}
Suppose that $f,g \in\Seq$ are sequences, $X,Y\in \ZZ[q,q^{-1}]$ are Laurent polynomials and $C\in \ZZ$ is an integer.
Then the operators $\Do_X$ satisfy the properties below. 
  \begin{enumerate}
\item $\tR$-linearity: $\Do_X(Yf) = Y\Do_X(f)$  and $\Do_X(f + g) = \Do_X(f) + \Do_X(g)$
\item Derivation: $\Do_X(fg) = \Do_X(f)g + (X(Lf)) \Do_1(g)$
\item Commutation: $\Do_X\Do_Y = \Do_Y\Do_X$
\item Annihilation: $\Do_X(q^{Cn}) = q^{Cn}(1-Xq^{C})$. In particular, if $X=q^{-C}$ then $\Do_X (q^{Cn}) = 0$.
  \end{enumerate}
  \end{prop}
\begin{proof}
  All of the properties are verified directly from the definition of $\Do_X$.
  \begin{enumerate}
  \item For $\tR$-linearity,
\begin{align*}
  \Do_X(Yf)(n) &= Yf(n) - YX f(n+1)\\
  &= Y(f(n) - X f(n+1))\\
  &= Y \Do_X(f)(n)
 \end{align*}
  \begin{align*}
    \Do_X(f + g)(n) &= (f+g)(n) - X (f+g)(n+1)\\
    &= f(n) + g(n) - X (f(n+1) + g(n+1))\\
    &= (f(n) - X f(n+1)) + (g(n) - X g(n+1))\\
    &= \Do_X(f)(n) + \Do_X(g)(n)
    \end{align*}
\item Derivation,
  \begin{align*}
    \Do_X(fg)(n) &= f(n)g(n) - Xf(n+1)g(n+1) \\
    &= (f(n) - Xf(n+1))g(n) + Xf(n+1)(g(n) - g(n+1))
    \end{align*}
\item For commutation,
\begin{align*}
  (\Do_X\Do_Yf)(n) &= (\Do_Yf)(n) - X(\Do_Yf)(n+1)\\
  &= (f(n)-Yf(n+1)) - X(f(n+1) - Yf(n+2))\\
  &= f(n) - (X+Y) f(n+1) + XY f(n+2)
    \end{align*}
which is symmetric in $X$ and $Y$, so $(\Do_X\Do_Yf)(n) = (\Do_Y\Do_Xf)(n)$.
\item Annihilation: $\Do_X(q^{Cn})(n) = q^{Cn} - X q^{C(n+1)} = q^{Cn}(1-Xq^C)$
    \end{enumerate}
  \end{proof}

The lemma  below is a consequence of the properties from the Prop. \ref{dpropprop} above.

\newcommand{\clex}{\alex}
\newcommand{\theA}{\widebar{A}}
\newcommand{\thecatA}{\widebar{\fA}}

\begin{lemma}\label{knotlem}
  Given a knot $K$ in $S^3$, if the Alexander polynomial of $K$ is written as
  $$\alex^1_K(q) = \sum_{i=1}^N a_i q^{n_i} \conj { where } a_i \in\ZZ\backslash\{ 0\}, n_i \textnormal{ distinct }$$
  then the operator $\tA_K:=\prod_i \Do_{q^{-n_i}}$  annihilates the sequence of reduced Alexander polynomials  $\clex_K(r):= \sralex_K(q) = \alex^1_K(q^r)$
  \end{lemma}
\begin{proof}
  Induction on the number of non-zero terms $N$. If $N=1$ then $\alex_K(q) = a_1q^{n_1}$ and $\Do_{q^{-n_1}}$ annihilates the associated sequence, by Prop. \ref{dpropprop} (4). Suppose $\alex_K(q) = f(q) + aq^n$  where $f(q) := \sum_{i=1}^{N-1} a_i q^{n_i}$, with $a_i \neq 0$ for $i=1,\ldots,N-1$. By induction, the operator
  $\Do' := \Do_{q^{-n_1}} \Do_{q^{-n_2}} \cdots \Do_{q^{-n_{N-1}}}$ annihilates the sequence $r\mapsto f(q^r)$. So
  \begin{align*}
(\tA_K\clex_K)(r)   &= ((\Do'\Do_{q^{-n}})\clex_K)(r) & \\
&= \Do'(\Do_{q^{-n}}(f(q^r) + aq^{rn})) & \\
    &= \Do'\Do_{q^{-n}}f(q^r) + \Do'(a\Do_{q^{-n}}(q^{rn})) & (\ref{dpropprop}\,\, (1))\\
    &= \Do'\Do_{q^{-n}}f(q^r) + \Do'(a\cdot 0) & (\ref{dpropprop}\,\, (4))\\
    &= \Do'\Do_{q^{-n}}f(q^r) &\\
    &= \Do_{q^{-n}}\Do'f(q^r)& (\ref{dpropprop}\,\, (3))\\
    &= \Do_{q^{-n}}(0) &\\
&= 0 &
    \end{align*}
  \end{proof}

\begin{example}
The operator $\Do_{1}\Do_{q^{-1}} \Do_{q^{-2}}$ annihilates the sequence $\alex_{3_1}(q^r) = 1 - q^r + q^{2r}$ of reduced $S^r$-colored Alexander polynomials  from Eqn. \eqref{redeq} associated to the right-handed trefoil $K=3_1$.
  \end{example}

\subsubsection{Incorporating the unknot $U$}\label{unknotsec}
Recall from Eqn. \eqref{unknoteq} that the $S^r$-colored Alexander polynomial of the unknot $U$ can be identified with the power series expansion of the rational function
$$\unsralex_U(q) = \frac{q-1}{q^r-1} = (1-q)\sum_{i=0}^{\infty} q^{ri}$$
\begin{lemma}\label{unknotlem}
If $\unalex_U(r) := \unsralex_U(q)$ defines the sequence $\unalex_U\in \Seq$ then 
 $$(M-1)(L-1)\unalex_U = 0$$
\end{lemma}
\begin{proof}
The unknot sequence is $\unalex_U(r) = (q-1)/(q^r-1)$. So
$$(q^r-1)\unalex_U(r) - (q^{r-1}-1)\unalex_U(r-1) = 0$$
which can be rewritten as $(M-1)(L-1)\unalex_U = 0$.
\end{proof}

Now for each knot $K$ in $S^3$, we would like a recurrence relation for the sequence of unreduced $S^r$-colored Alexander polynomials from Def. \ref{coloredalexdef} and Eqn. \eqref{alexseq}.
$$\unalex_K(r) := \unsralex_K(q) = \alex_K(r) \unalex_U(r)$$
Lemmas \ref{knotlem} and \ref{unknotlem} gave recurrence relations, $\tA_K$ and $(M-1)(L-1)$, for the sequences $\alex_K(r)$ and $\unalex_U(r)$ respectively. The existence of a recurrence relation for $\unalex_K(r)$ now follows from the general theory, see \cite[Thm. 5.2 (b)]{GLsurvey} or \cite[Cor. 2.1.6]{Sabbah}. However, we are interested in explicit formulas so a direct calculation is included below.

\begin{thm}\label{alexqholthm}
The element $\theA_K := (M-1)\tA_K \in \weyl$ annihilates the sequence $\unalex_K\in \Seq$ of unreduced $S^r$-colored Alexander polynomials defined by $\unalex_K(r) := \unsralex_K(q)$.
  \end{thm}
\begin{proof}
Define the coefficients $a_j$ by the relation
$$\tA_K \blex_K(r) = \sum_{j=r}^{r+d} a_j \blex_K(j)$$
Lemma \ref{knotlem} shows $\tA_K \blex_K=0$, so the sum $\sum_{j=r}^{r+d} a_j \blex_K(j)=0$ is zero.
Multiplying through by $\unalex_U(r+d)$ shows that
 $\sum_{j=r}^{r+d} a_j \unalex_U(r+d) \blex_K(j) = 0$.
Using the recurrence 
$\unalex_U(r+d) = (q^{n+j}-1)/(q^{n+d}-1)\unalex_U(r+j)$
from Lemma \ref{unknotlem} for each term this sum gives 
$\sum_{j=r}^{r+d} a_j \frac{q^{n+j}-1}{q^{n+d}-1} \unalex_U(r+j) \blex_K(j) = 0$. Multiplying both sides of this sum by $q^{n+d}-1$ gives us the relation $\sum_{j=r}^{r+d} a_j (q^{n+j}-1) \unalex_U(r+j) \blex_K(j) = 0$.
Simplifying further,
\begin{align*}
0= \sum_{j=r}^{r+d} a_{j} (q^{n+j}-1) \blex_K(r+j) \unalex_U(r+j) 
&= (M-1)\sum_{j=r}^{r+d} a_{j} \blex_K(r+j) \unalex_U(r+j) \\
&= (M-1)\tA_K \blex_K\unalex_U \\
&=\theA_K \unalex_K
\end{align*}
  \end{proof}

\section{Overview of Heegaard-Floer theory}
\label{sec:HFKoverview}
\newcommand{\og}{\overline{\gamma}}
\newcommand{\hog}{\frac{1}{2}\overline{\gamma}}
\newcommand{\oga}{\og}
\newcommand{\omu}{\overline{\mu}}
This section contains a review of the Heegaard Floer materials which are
needed for Section \ref{altsec}. In particular, we recall the bordered
Heegaard Floer theory, its relation to the immersed curves interpretation
of $\widehat{\textit{CFD}}(S^3\backslash \nu K)$ and how to compute $\hfk(K)$ with its
bigrading. The last subsection describes the cabling algorithm for immersed curves.

Heegaard Floer homology is a suite of closed 3-manifold invariants which was
introduced by Ozsv\'ath and Szab\'o \cite{OS03,OS03prop}. 
For a closed 3-manifold $Y$ and a
$\spinc$ structure $\mathfrak{s}$, there are relatively graded $\FF_2$-vector space valued invariants $\textit{HF}^{\infty}(Y, \mathfrak{s}), \textit{HF}^{\pm}(Y, \mathfrak{s}),$ and $\widehat{\textit{HF}}(Y,\mathfrak{s})$ of the pair $(Y,\mathfrak{s})$. While these invariants are specified by combinatorial data, called $w$-pointed Heegaard diagrams, they measure subtle topological information about the manifold $Y$ using pseudo-holomorphic curves.

To a knot $K$ in a $3$-manifold $Y$, one can associate a ($z, w$)-doubly pointed
Heegaard diagram. Topologically, the two marked points and handle decomposition
specify the knot $K$ within $Y$. Algebraically, the extra $z$ marked point induces a filtration on
the Heegaard Floer chain complex associated to the $w$-pointed diagram. When $Y=S^3$, this
results in the bigraded knot Floer homology $\hfk(K)$, \cite{OZ} and \cite{Ras03}, which
categorifies the Alexander polynomial in the sense that
$$\alex_K(q) = \sum_{i,j} (-1)^i q^j \dim \hfk^{i,j}(K).$$

The knot Floer invariants $\hfk(K)$ can be difficult to compute from a
doubly pointed Heegaard diagram, see \cite{Hom20}. However, Hanselman,
Rasmussen and Watson \cite{HRW1} reinterpreted the bordered Floer invariants
and the pairing theorem \cite{LOT18b} by identifying $\hfk(K)$ with a chain
complex of morphisms between immersed curves in the Fukaya category of a
surface. It is this simpler setting which facilitates our study in Section
\ref{altsec}. The remainder of this section contains a review of these ideas
and their essential properties.

\subsection{The curve invariant}
\label{ICs}
  
Bordered Heegaard Floer homology, introduced by Lipshitz, Ozsv\'ath, and
Thurston \cite{LOT18b}, is a relative version of Heegaard Floer homology.
The bordered theory associates to a compact connected surface with one boundary component, $\Sigma$,
a dg algebra $\mathcal{A}$. If $M$ is a compact
$3$-manifold with boundary suitably\footnote{The theory is defined in terms of manifolds with fixed handle
decompositions; we are glossing over these kinds of details here.} parameterized by a map $\phi : \Sigma \to M$ then there
are two bordered invariants related to $(M, \phi)$: a type $D$ structure
$\widehat{\textit{CFD}}(M, \phi)$, that is a left dg module over
$\mathcal{A}$, and a type $A$ structure $\widehat{\textit{CFA}}(M, \phi)$, 
that is a right $\mathcal{A}_{\infty}$-module over $\mathcal{A}$.  When $Y$
is a closed $3$-manifold that is obtained by gluing two boundary parameterized $3$-manifolds, $(M,\psi)$ and
$(N,\phi)$, together using $h:=\psi^{-1}\circ\phi$
\begin{gather}\label{eq:pairing}
  Y = M \sqcup_h N  \Rightarrow\nonumber\\
  \widehat{HF}(Y) \cong H_*(\widehat{\textit{CFA}}(M, \psi) \boxtimes \widehat{\textit{CFD}}(N, \phi))
\end{gather}
there is a product $\boxtimes$ of these modules which recovers the Heegaard-Floer invariant of $Y$ in the sense of Eqn. \eqref{eq:pairing}, see \cite[\S 9]{LOT18b}. 

For $3$-manifold with torus boundary, a parameterization $\phi$ is equivalent to choosing a pair of curves $(\alpha,\beta)$ of $\partial M$ and fixing a basepoint $w \in \partial Y$.  For knot complements $M = S^3 \setminus \nu K$, the parameterization $\phi$ will always be described by the Seifert-framed meridian-longitude basis $\left\{\mu, \lambda\right\}$.

Hanselman, Rasmussen and Watson \cite{HRW1,HRW2}, also \cite[\S 4.4]{HKK},
show that type $D$ structures $\widehat{\textit{CFD}}(M, \phi)$ are
equivalent to immersed curves $\gamma_M$ (elsewhere denoted by
$\widehat{\textit{HF}}(M)$ or $\hat{\Gamma}(M)$) in the surface $T_M =
\partial M \, \setminus w$, decorated by local systems and defined up to
regular homotopy. For a knot complement $M = M_K$ where $M_K := S^3\backslash \nu K$, set
$\gamma_K := \gamma_{M_K}$.

\begin{remark}{(Local systems)}
Knot complements $M = S^3\backslash \nu K$ are of \textit{loop type}
\cite{HW23b}, as a consequence the associated immersed curves $\gamma_K$ 
have trivial local systems.
\end{remark}

\begin{remark}{(Grading arrows)}\label{rem:grarrow}
If the invariant $\gamma_K$ has multiple components then they are
connected by pairs of edges called  grading arrows \cite[Definition 28]{HRW2}. While domains involving grading arrows do not
contribute to the differential, they are considered when determining Maslov
grading differences, as in Prop. \ref{prop:grading}.
\end{remark}

The collection of immersed curves $\ga_M$ in $T_M$ is almost always lifted to a curve $\og_M$ in the covering space $\overline{T}_M$ described next.

\begin{definition} 
\label{def:curvecover}
The covering space $\overline{T}_M$ of $T_M = \partial M \setminus \{w\}$ is the cover whose fundamental group is the kernel of the composition
\[
\pi_1(T_M) \rightarrow \pi_1(\partial M) \rightarrow H_1(\partial M) \rightarrow H_1(M).
\]
When $H_1(M) \cong \mathbb{Z}$, this cover $\overline{T}_{M}$ is homeomorphic to an infinite cylinder 
 with $z$ lifting to infinitely-many marked points at heights of $\frac{1}{2}+\mathbb{Z}$. The meridian $\mu$ lifts to a line
$\omu$ which passes through these lifted marked points (See Figure \ref{fig:RHTcurve}). In this paper, the infinite cylinder will always be pictured as a quotient of a horizontal infinite rectangle, and the heights of the marked points will be are lower to the left and higher to the right. 

A curve $\og_0$ which {\em ``wraps around the cylinder''} is a homologically essential simple closed curve. When the curve is oriented or the two quotient lines are ordered then the points of intersection between $\og_0$ and other curves inherit an order. 
\end{definition}

\begin{notation}
\label{note:spunc}
For a $3$-manifold $M$ with torus boundary, a lift of the immersed curve
invariant $\gamma_M$ to the covering space $\overline{T}_M$ typically
depends on a $\spinc$ structure of $M$ and so must be denoted by
$\og_{(M,\mathfrak{s})}$ (or $\hf(M, \mathfrak{s})$). However, there is a
unique $\spinc$ structure $\mathfrak{s}_0$ on knot complements $M_K :=S^3 \setminus \nu K$ because $H^2(M_K)\cong H_1(M_K,\partial M_K) = 0$. So we
will suppress $\mathfrak{s}_0$ from the notation $\og_K$. We also adopt the
notations $T_K \defeq T_{M_K}$ and $\overline{T}_K \defeq
\overline{T}_{M_K}$ unless otherwise stated.
\end{notation}

\begin{figure}
\begin{tikzpicture}
\draw[gray, very thin, dashed, step=1cm] (0,.5) grid[ystep=0] (2,2.5);
 \draw (0,.125) node {$-1$};
 \draw (1,.125) node {$0$};
 \draw (2,.125) node {$1$};

\draw[thick, blue] (0,1.5) -- (2,1.5);
 \draw[thick, red] (1,2.5) -- (1,.5);
\draw (-6pt,1.5) node {$\omu$};
\draw (1,2.75) node {$\og$};

\foreach \x in {0.5,1.5}
  \draw (\x,1.5) node {$\bullet$};

\draw[ultra thick,->] (0,2.5) -- (2,2.5);
\draw[ultra thick,->] (0,.5) -- (2,.5);

\draw[gray, very thin, dashed, step=1cm] (5,.5) grid[ystep=0] (9,2.5);
\draw[thick,blue] (5,1.5) -- (9,1.5);
  \draw (5,0.125) node {$-2$};
  \draw (6,0.125) node {$-1$};
  \draw (7,0.125) node {$0$};
  \draw (8,0.125) node {$1$};
  \draw (9,0.125) node {$2$};


  \draw[thick,red] (7,.5)
  .. controls +(up:.5cm) and +(down:0.3cm) .. (6,1.5)
  .. controls +(up:.3cm) and +(left:0.3cm) .. (6.5,2)
 .. controls +(right:.3cm) and +(left:0.3cm) .. (7.5,1)
 .. controls +(right:.3cm) and +(down:0.3cm) .. (8,1.5)
 .. controls +(up:.3cm) and +(down:.5cm) .. (7,2.5);

\foreach \x in {5.5,6.5,7.5,8.5}
  \draw (\x,1.5) node {$\bullet$};
\draw[ultra thick,->] (5,2.5) -- (9,2.5);
\draw[ultra thick,->] (5,.5) -- (9,.5);

\draw (4.8,1.5) node {$\omu$};
\draw (7,2.75) node {$\og$};

  \end{tikzpicture}
\caption{The invariants for the unknot $U$ and right-handed $T(2,3)$, lifted to $\overline{T}_M$. Integral heights between the marked points $\bullet$ are indicated underneath.}

\label{fig:RHTcurve}
\end{figure}
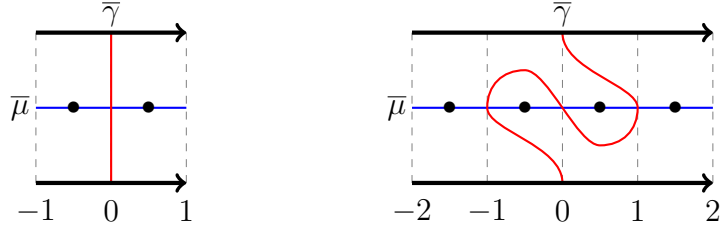

The curve valued knot invariant
$K\mapsto \og_K$ in $\overline{T}_K$ has a few important 
structural properties which we list below.

\begin{remark}{(Structural properties)}\label{rem:strprop}
\begin{enumerate}
\item  The invariant $\og_K$ contains an {\em essential component}
$\og_0$ that wraps around the cylinder $\overline{T}_K$.
All of the other components of $\og_K$ are closed, homologically inessential and lie within an $\epsilon$-neighborhood of the curve $\omu$ \cite[Corollary 64]{HRW2}.
\item The invariant $\gamma_K$ in $T_K$ is invariant under hyperelliptic involution of $T_K$ and so its lift $\og_K$ to $\overline{T}_K$ is unchanged under rotation by $\pi$ \cite[Theorem 7]{HRW2}. As a consequence, we always take the lift of $\og_K$ to be centered at height $0$. 
\end{enumerate}
\end{remark}

\begin{remark}{(Knot invariants)}\label{rem:kinv}
\begin{enumerate}
\item The Seifert genus $g(K)$ is the largest height in $\overline{T}_K$ for which any component of
$\og_K$ intersects $\omu$ \cite[Proposition 48]{HRW2}. 
\item The invariant $\tau(K)$ is the height in $\overline{T}_K$ at which
  $\og_0$ first intersects $\omu$ after ``wrapping around the cylinder.''
\end{enumerate}
\end{remark}

\newcommand{\CF}{\textit{CF}}
\newcommand{\HF}{\textit{HF}}

The pairing theorem for bordered Heegaard-Floer invariants from
Eqn. \eqref{eq:pairing} admits an immersed curves interpretation too. The
immersed curves $\gamma_M$ or $\oga_M$ associated to type D structures are naturally
viewed as objects in the Fukaya category of $T_M$ or $\overline{T}_M$ respectively. In this category, 
there is a chain complex $(\CF(\gamma, \eta), d)$ of morphisms called {\em intersection Floer theory}.
As an $\FF_2$-vector space, $\CF(\gamma, \eta)$ is generated by the transverse points of intersection
$\gamma \pitchfork \eta$ between the two curves. The differential $d$ counts
bigons such as those in Figure \ref{fig:commonbigons}, see \cite[\S 2]{Abouzaid} or \cite[\S 4.1]{HRW1}.
The theorem below restates the pairing theorem from Eqn. \eqref{eq:pairing} in these terms.

\begin{theorem}[{\cite[Theorem 2]{HRW2}}]\label{thm:ICpairing}
 Suppose that a closed $3$-manifold $Y$ is homeomorphic to a quotient
 $Y\cong M \sqcup_h N$, where the $M$ and $N$ are compact, oriented
 $3$-manifolds with torus boundary and $h : \partial N \rightarrow \partial M$ is an orientation reversing homeomorphism. Then there is
 an isomorphism between the Heegaard-Floer homology of $Y$ and the
 intersection Floer homology in $T_{M}$.
\[
\hf(M \sqcup_h N) \cong H_*(\CF(\gamma_M, h(\gamma_N)))
\]

This isomorphism respects relative gradings and $\spinc$ structures.
\end{theorem}

\subsection{Knot Floer Homology}\label{sec:knotfloer}
\newcommand{\hC}{\hat{C}}
\newcommand{\Cm}{C^{-}}
\newcommand{\gCm}{gC^{-}}

Knot Floer homology $\hfk(K)$ for a knot $K\subset S^3$ is computed from a Heegaard diagram with two basepoints, $w$ and $z$ \cite[\S 3.1]{OZ}. Such a diagram is constructed by upgrading single $w$-pointed
Heegaard diagram for $S^3$ to a double $(w,z)$-pointed Heegaard diagram.
After adding the extra $z$-basepoint, the Maslov grading, sensitive to bigons covering the $w$-point, is retained and a second {\em Alexander grading} is introduced which counts bigons covering the new $z$-point, see \cite{Hom20}.

The pairing Theorem \ref{thm:ICpairing} above allows us to introduce 
knot Floer homology $\hfk(K)$ in Theorem \ref{thm:IChfk} below. In this setting, we want to
glue the knot complement $M=S^3\backslash \nu K$ with the $\infty$-framed solid torus.
This corresponds to a pairing between  the curve $\og_K$ and the meridian
curve $\omu$.  Without any modifications to Thm. \ref{thm:ICpairing}, this pairing produces a chain complex
which computes only the $w$-pointed homology $\hf(S^3)$ of the gluing. 

In order to recover knot Floer homology $\hfk(K)$ and it's Alexander
filtration, the covering space $\overline{T}_K$ is modified by replacing the
lifts of the marked point $w$ 
at heights $(\frac{2n+1}{2},0)$ with small disks $D_n$, (with radius small
enough not to intersect the curve $\og_K$). Within each disk $D_n$ are lifts
of marked points $z$ and $w$ at heights $(\frac{2n+1}{2}, \epsilon)$ and
$(\frac{2n+1}{2},-\epsilon)$, respectively.  The meridian $\omu$ is now the
curve that goes between consecutive pairs of marked points.

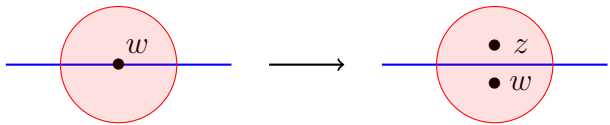
\begin{figure}[!h]
\centering
\begin{tikzpicture}

\draw[thick, blue, -] (0,3) -- (3,3);
  \draw (1.5,3) node {$\bullet$};
  \draw (1.5+.25,3.25) node {$w$};
\filldraw[red,fill opacity=0.125] (1.5,3) circle (22pt);

\draw[thick,->] (3.5,3) -- (4.5,3);

\draw[thick, blue, -] (5,3) -- (8,3);

  \draw (5+1.5,3.25) node {$\bullet$};
  \draw (5+1.85,3.25) node {$z$};
  \draw (5+1.5,3-.25) node {$\bullet$};
  \draw (5+1.85,3-.25) node {$w$};
\filldraw[red,fill opacity=0.125] (5+1.5,3) circle (22pt);

\end{tikzpicture}
\caption{To construct $C^-$ $w$-marked points are replaced with $(z,w)$-pairs of marked points}
\label{fig:markedptenlargement}
\end{figure}

\newcommand{\vdiff}{\partial^v}

To reconstruct the knot Floer homology, begin by setting
$$\Cm(\oga_M, \omu) := \mathbb{F}_2[U]\otimes_{\mathbb{F}_2} \CF(\oga_M,\omu).$$
This is generated over $\mathbb{F}[U]$ by the intersection points $\oga_M \pitchfork \omu$ between $\oga_M$  and $\omu$. 
The additional $z$-basepoints added in the previous paragraph allow us to count bigons in new ways.
If $B$ is a bigon then denote by $n_w(B)$ and $n_z(B)$ the number of times $B$ covers $w$ and $z$ respectively.
Let $N_i^w(x,y)\in\mathbb{F}_2$ be the number of bigons $B$ from $x$ to $y$ for which $n_w(B)=i$.
The differential 
$$d : \Cm(\oga_M,\omu) \to \Cm(\oga_M,\omu) \conj{ is } d(x) := \sum_{i=0}^{\infty} \sum_{y} U^{i}N_{i}^{w}(x,y) \cdot y.$$
Without modification this construction does not recover knot Floer complex $\textit{CFK}^-(K)$. There is a homotopy equivalence \cite[Warning \S 4.2]{HRW2} between associated graded objects
$$\gr C^-(\og_M, \omu) \simeq \gr \textit{CFK}^-(K)$$
with respect to the Alexander filtration discussed below. The issue is that $C^-(\og_M, \omu)$ is not necessarily a chain complex because of the existence of teardrop or fishtail shaped regions. To obtain a chain complex $C^-(\og_M,\omu)$ which is equivalent to $\textit{CFK}^-(K)$ one must deform the construction using a bounding chain that counts these teardrop-shaped regions \cite[Def. 3.12]{Hanselman23}.

We can recover $\hfk(K)$ without invoking this extension. The hat-version of the $\Cm$-construction is obtained by setting $U=0$.
\begin{equation}\label{eq:hatc}
  (\hC(K),\vdiff) := (\Cm(\og_K,\omu), d)\otimes_{\FF_2[U]} \FF_2
\end{equation}  
This is always a chain complex because it agrees with the $U=0$ specialization of the deformed $C^-$ mentioned above\footnote{This is because when $U=0$, the differential of the deformed $C^-$-complex defined by Hanselman no longer counts teardrop-shaped regions when they cover a basepoint (as they cover a $w$-marked point) and any teardrop-shaped regions which do not cover a basepoint can be removed by homotopy invariance. So the contribution of the bounding chain to the differential at $U=0$ is zero. }.
The homology of the chain complex $\hC(K)$ computes knot Floer homology $\hfk(K)$ in the sense of Thm \ref{thm:IChfk} below. Before stating this theorem, we recall the filtration, the Alexander grading and the Maslov grading.

The {\em Alexander grading} $A$, or the $q$-grading, is computed using the intersection numbers $n_z$ and $n_w$ of bigons with the $z$ and $w$-basepoints. In more detail, after choosing the symmetric presentation in Remark \ref{rem:strprop} (2),
there is a function $A$ on generators $\og_K\pitchfork \omu$ of $\Cm(\og_M, \omu)$. If $B$ is a bigon from $x$ to $y$ then 
$$A(x)-A(y)=n_z(B)-n_w(B) \conj{ and } A(U \cdot x) = A(x)-1.$$
There is a {\em filtration} $F^n \Cm(\og_M,\omu) := \{ x : A(x)\leq n \}$; the associated graded $\gr\Cm(\og_M, \omu)$ is the chain complex which is obtained by considering only bigons $B$ with $n_z(B)=0$.

The {\em Maslov grading} $M$ is the $t$-grading or homological grading. Suppose that 
$x$ and $y$ are two points of intersection between $\og_K$ and $\omu$, and $B$ is a bigon from $x$ to $y$. So the boundary of $B$
consists of two piecewise smooth paths one from $x$ to $y$ in $\og_K$ and the other from $y$ to $x$ in $\omu$. 
The relative Maslov gradings can be computed using the equation,
\begin{equation}\label{eq:grformula}
M(y)-M(x) = 2\left(\text{Wind}_w(B) + \text{Wght}(B) - \text{Rot}(B)\right).
\end{equation}
see  \cite{Han23a}. In this formula, $\text{Wind}_w(B)$ is the net winding number of $B$ around enclosed $w$ basepoints. The other two terms involve grading arrows, see Rmk. \ref{rem:grarrow} or \cite[\S 2.1]{Han23a}, $\text{Wght}(B)$ is the sum of weights (counted with sign) of all grading arrows traversed by $B$ and
$\text{Rot}(B)$ is $\frac{1}{2\pi}$ times the total counterclockwise rotation along the smooth
sections of $B$. Alternatively,
$$\text{Rot}(B)=\frac{1}{2\pi}(2\pi - a\frac{\pi}{2} - c\pi),$$ 
where  $a$ and $c$ are the numbers of corners and cusps (from grading arrows) traversed respectively.

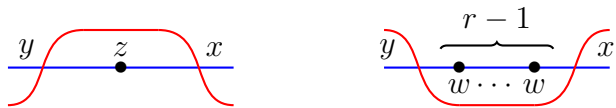
\begin{figure}[!h]
\centering
\begin{tikzpicture}

\draw[thick, blue, -] (0,3) -- (3,3);
  \draw (1.5,3) node {$\bullet$};
  \draw (1.5,3.25) node {$z$};

  \draw[thick,red] (0,2.5) .. controls +(right:.6cm) and +(left:.7cm) .. (1,3.5);
  \draw[thick,red,-] (1,3.5) -- (2,3.5);
  \draw[thick,red] (2,3.5) .. controls +(right:.7cm) and +(left:.6cm) .. (3,2.5);

  \draw (0.25,3.25) node {$y$};
  \draw (3-.25,3.25) node {$x$};

\draw[thick, blue, -] (5,3) -- (8,3);

  \draw (5+1,3) node {$\bullet$};
   \draw (5+2,3) node {$\bullet$};

  \draw (5+1,3-.25) node {$w$};
  \draw (5+1.5,3-.25) node {$\cdots$};
  \draw (5+2,3-.25) node {$w$};

\draw [thick, decoration={brace}, decorate]  (5+.75,3.25) -- (5+2.25,3.25) node [pos=0.5,anchor=north,yshift=0.7cm]   {$r-1$};
  \draw (5+0.25-0.2,3.25) node {$y$};
  \draw (5+3-.25+0.2,3.25) node {$x$};

  \draw[thick,red] (5+0,3.5) .. controls +(right:.6cm) and +(left:.7cm) .. (5+1,2.5);
  \draw[thick,red,-] (5+1,2.5) -- (5+2,2.5);
  \draw[thick,red] (5+2,2.5) .. controls +(right:.7cm) and +(left:.6cm) .. (5+3,3.5);

\end{tikzpicture}

\caption{Bigons from $x$ to $y$ and from $y$ to $x$, see Ex. \ref{ex:bigons}}
\label{fig:commonbigons}
\end{figure}

\begin{example}{(Maslov gradings)}\label{ex:bigons}
Figure \ref{fig:commonbigons} shows two important bigons for us. Each black
circle is understood to have $z$ and $w$ marked points above and below,
respectively. For both bigons, $\text{Wght}(B)=0$ since there are no grading
arrows traversed by the bigons. The bigon on the left satisfies
$\text{Rot}(B) = \frac{1}{2}$ and bigon on the right satisfies
$\text{Rot}(B) = -\frac{1}{2}$; the sign appears because the boundaries of
the bigons have opposite orientations.  For the bigon on the left, there are
no $w$-points so Eqn. \eqref{eq:grformula} gives $M(y)-M(x)=1$.  For the
bigon on the right, there are $(r-1)$-many $w$-points and we have $M(y)-M(x)
= 2r-3$.
\end{example}

Once the chain complex $\hC(K)$ has been given the structure above, the theorem below can be stated more precisely.

\begin{thm}(\cite[Thm. 52]{HRW2})\label{thm:IChfk}
The filtered complex $(\hC(K),\vdiff)$ is homotopy equivalent to the filtered complex $\cfk(K)$ associated to a knot by \cite{OZ} and \cite{Ras03}. In particular, the associated graded complex computes the knot Floer homology
$$\hfk(K) \cong H_*(\gr \hC(K))$$
\end{thm}

Now that we understand how to study knot Floer homology of a curve using
immersed curve technology, the two remaining subsections provide essential
tools for understanding and manipulating the knot Floer homology of a curve.

\newcommand{\dgen}{\chi_0}
\newcommand{\bgen}{b}
\subsection{Vertical Simplification}\label{sec:vertsimp}

In this section we introduce vertically simplified bases. This will give us a reasonable toolkit for understanding the structure of $\hC(K)$ from Eqn. \eqref{eq:hatc} and it's relationship to the curve $\og_K$. 

\begin{definition}{(Vertically simplified basis)}
A basis $\{x_i\}$ for a chain complex $(C,d)$ is
{\em vertically-simplified} when for each basis element $x_i$, exactly one of
the following conditions holds:
\begin{enumerate}
\item $d x_i = x_j$ for some $x_j$,
\item $x_i$ is in the kernel but not the image of $d$,
\item $x_i = d x_j$ for a unique $x_j$.
\end{enumerate}
\label{def:vertsimpCFK}
\end{definition}

\begin{remark}\label{rmk:vertsimp}
  \begin{enumerate}
\item Vertical simplification is equivalent to 
 choosing an isomorphism of the form
  $$C\cong  \left(\oplus_i \Lambda(x_i) \right) \oplus D,$$
where $d$ vanishes on $D$, $\Lambda(x_i):= \FF_2\inp{1_i,x_i}$ and $d(x_i) = 1_i$.
\item A bicomplex $(C,d,\delta)$ can be vertically-simplified with respect to $d$ or with respect to $\delta$. In general, it is unknown whether all of the bicomplexes appearing in knot Floer theory can be simplified with respect to both differentials simultaneously.
\end{enumerate}
\end{remark}

We now fix some notations related to the vertical simplification of $\hC(K)$.

\begin{notation}\label{notation:vertsimp}
\begin{enumerate}
\item For each knot $K$, the chain complex $\hC(K)$ admits a vertically-simplified basis. Since $H_{\ast}(\cfk(K), \partial^{\text{v}}) \cong \hf(S^3) \cong \mathbb{F}_2\inp{\dgen}$, the subspace $D$ above is $1$-dimensional.  There 
are isomorphisms
\begin{equation}\label{eq:vertsimp}
  \hC(K) \cong \FF_2\inp{\dgen} \opp \bigoplus_{i\in I} \Lambda(x_i)
\end{equation}  
where $\Lambda(x_i) = \FF_2\inp{1_i, x_i}$ and $\partial^{\text{v}}(x_{i}) = 1_i$.
\item The set $I$ is ordered so that $j \geq i$ implies $A(x_{j}) \leq A(x_{i})$. 
\item The {\em distinguished generator} $\dgen$ generates $\hf(S^3)$ under the identification from Theorem \ref{thm:IChfk}.
\item There is a {\em bottom generator} $\bgen$  which has the least Maslov grading among generators $y$ which satisfy $A(y) = -g(K)$.
  \end{enumerate}
  \end{notation}

\begin{remark}\label{rem:geodesic}
A vertically-simplified basis ensures that no two generators are connected
by a differential that does not reduce the Alexander grading, such a basis is called {\em reduced}.
Equivalently, there are no bigons in the differentials of $\hC(K)$ or $\Cm(\og_K, \overline{\mu})$
that do not cross marked points. Note that, homotopy invariance of $\og_K$ allows us to assume each curve segment between generators is length-minimizing. This also produces a reduced chain complex.
\end{remark}

If a chain complex $\textit{CFK}^-(K)$ which has been simultaneously horizontally and vertically simplified  then 
the algorithm below tells us how to reconstruct the immersed curve $\og_K$. This proposition is the immersed curves reformulation of \cite[Theorem 11.31]{LOT18b}.

\begin{proposition}{(\cite[Proposition 47]{HRW2})} 
Given a basis for $\textit{CFK}^-(K)$ which is both horizontally-simplified and vertically-simplified, the curve invariant $\oga_K$ is obtained in $\overline{T}_K$ by:
\begin{enumerate}
\item For each basis element $x$ of $\textit{CFK}^-(K)$, place a short verticle segment at height $A(x)$.
\item If $\textit{CFK}^-(K)$ contains a vertical arrow from $x$ to $y$, then connect the top endpoints of their corresponding short segments by an arc.
\item If $\textit{CFK}^-(K)$ contains a horizontal arrow from $x$ to $y$, then connect the bottom endpoints of their corresponding short segments by an arc.
\item There will remain a unique segment with an unattached top endpoint and a unique segment with an unattached bottom endpoint. 
Attach the top endpoint to the top edge of $\overline{T}_M$ at height $0$ and  attach the remaining bottom endpoint to the bottom edge of $\overline{T}_M$ at height $0$.
\end{enumerate}
\label{prop:HFfromCFK}
\end{proposition}

As noted in Rmk. \ref{rmk:vertsimp} (2), we cannot assume that such a simultaneously simplified basis exists. However, by 
ignoring step (3) and the latter half of step (4), we can obtain partial information about $\og_K$ discussed below (that is used in Section \ref{altsec}). The next corollary will play a role in Lemma \ref{lem:cablingstr}.

\begin{cor}\label{cor:structurecor}
If $K$ is a knot in $S^3$ and 
$\hC(K) \cong \FF_2\inp{\dgen} \opp \bigoplus_{i\in I} \Lambda(x_i)$
a vertically-simplified basis, with elements ordered as in Notation \ref{notation:vertsimp}, then there is a half-completed curve $\hog_K$ in $\overline{T}_K$ which satisfies
\begin{enumerate}
\item Each generator $y$ corresponds to a short vertical segment at height given by its Alexander grading $A(y)$ in $\overline{T}_K$.
\item A horizontal curve segment connects each pair $x_{i}$ and $1_{i}$ along their top endpoints corresponding to the length $n_i$ arrow in $\partial^{\text{v}}$.
\item The vertical segment corresponding to the distinguished generator $\dgen$ connects to the top edge of $\overline{T}_K$.
\item The vertical segment corresponding to the bottom generator $\bgen$ is a vertical segment of lowest height in $\overline{T}_K$.
\end{enumerate}
  \end{cor}

\subsection{Cabling immersed curves}
\label{sec:cablingcurves}

The Alexander polynomial transforms according to Eqn. \eqref{cableeq} under the cabling operation described by Def. \ref{def:cable}.
In \cite{HW23}, Hanselman and Watson discovered how the $(p,q)$-cable operation
transforms the immerse curves description of Heegaard-Floer homology described above.
If the curve $\oga_{K}$ is placed within a minimal box as in Fig. \ref{fig:pattern} then the $(3,4)$-cable $K_{3,4}$ of a knot $K$ 
leads to the immersed curve $\og_{K_{3,4}}$ in Figure \ref{fig:HWCabling}.

\begin{theorem}[\cite{HW23}] \label{thm:cabling}
Suppose that $\oga_K$ in $\overline{T}_K$ is the immersed curve associated to a knot $K$ in $S^3$. Then the procedure below produces the immersed curve $\oga_{K_{p,q}}$  associated to the $(p,q)$-cable of $K$\footnote{Compared to \cite{HW23}, all of our diagrams have been rotated by $-\pi/2$.}.
\begin{enumerate}
\item Create $p$ parallel copies of $\overline{T}_{K}$ stacked vertically. Place a copy of $\oga_K$ in each $\overline{T}_K$, so that 
they are scaled horizontally by a factor of $p$ and the next lower copy of the multicurve is $q$ units lower (in $A$-grading or height) than the previous copy
\item Connect the loose ends of the successive copies of the curve
\item Observe that the marked points (or disks) admit an isotopy which moves them uniformly vertically downward until they appear in order in the bottom copy of $\overline{T}_K$.  Using this isotopy to carry the multicurve from steps (1) and (2) produces the immersed curve  $\og_{K_{p,q}}$.
\end{enumerate}
\label{thm:cablecurve}
\end{theorem}

\begin{notation}
Step (3) allows us to view the resulting multicurve in the cover
$\overline{T}_{K_{p,q}}$, but it also makes a mess, so in what follows it will be 
convenient to skip this step. We instead use a single zig-zag curve $\omu$ through all unisotoped lifts
of the marked point as in Figure \ref{fig:HWCabling}. In particular, when considering the $(r,rn+1)$-cable $K_{r,rn+1}$ of $K$, by a slight abuse of notation, we will use $\oga_{K_{r,rn+1}}$ to denote the immersed curve resulting from steps (1) and (2) of the above construction and $\overline{T}^r$ to denote the ambient surface.
\end{notation}

\begin{notation}{($\circ$-dot)}\label{circlenotation}
Moreover, when illustrating the curves associated to cablings, it is
often useful to simplify the pictures by combining $\bullet$ components above the first row into a single $\circ$-component, see Fig. \ref{fig:dilation} or  Fig. \ref{fig:StableUnknot}. In later sections, focus will placed on this first row because of its non-trivial contribution to the $\scfk(K)$ construction.
\end{notation}

\begin{figure}
  \begin{tikzpicture}
\draw[gray, very thin, dashed, step=1cm] (0,.5) grid[ystep=0] (4,2.5);
\foreach \x in {.5,1.5,2.5,3.5}
  \draw (\x,1.5) node {$\bullet$};

  \draw[thick,red] (2,.5) -- (2,1);
  \draw[thick,red] (2,2.5) -- (2,2);
\draw[thick,red] (1,1) rectangle (3,2);
\filldraw[red,fill opacity=0.125] (1,1) rectangle (3,2);
\draw[ultra thick,->] (0,2.5) -- (4,2.5);
\draw[ultra thick,->] (0,.5) -- (4,.5);
\draw [thick, decoration={brace,mirror}, decorate]  (1,-8pt) -- (3-0.05,-8pt) node [pos=0.5,anchor=south,yshift=-0.75cm]   {$2g$};

  \end{tikzpicture}
\caption{A cabling pattern. A minimal box is $2g(K)$ wide.}
\label{fig:pattern}
\end{figure}
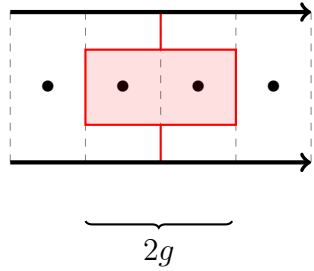

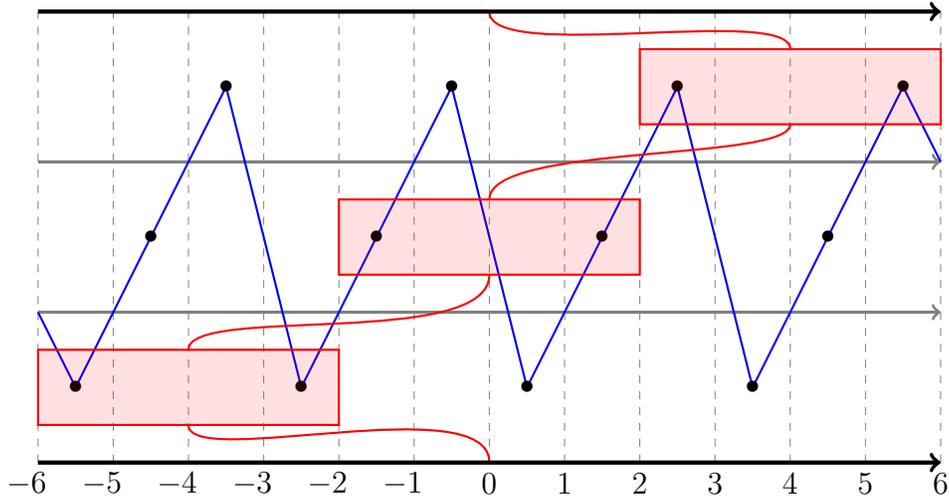
\begin{figure}[!h]
\centering
\begin{tikzpicture}
\draw[gray, very thin, dashed, step=1cm] (0,0) grid[ystep=0] (12,6);
\draw[very thick, gray, ->] (0,2) -- (12,2);
\draw[very thick, gray, ->] (0,4) -- (12,4);

\draw[thick,blue] (0,2) -- (.5,1) -- (1.5,3) -- (2.5,5) -- (3.5,1);
\draw[thick,blue] (3.5,1) -- (4.5,3) -- (5.5,5) -- (6.5,1);
\draw[thick,blue] (6.5,1) -- (7.5,3) -- (8.5,5) -- (9.5,1);
\draw[thick,blue] (9.5,1) -- (10.5,3) -- (11.5,5) -- (12,4);

\foreach \x in {.5,3.5,6.5,9.5}
  \draw (\x,1) node {$\bullet$};

\foreach \x in {1.5,4.5,7.5,10.5}
  \draw (\x,3) node {$\bullet$};

\foreach \x in {2.5,5.5,8.5,11.5}
  \draw (\x,5) node {$\bullet$};

\draw[thick,red] (0,.5) rectangle (4,1.5);
\draw[thick,red] (4,2.5) rectangle (8,3.5);
\draw[thick,red] (8,4.5) rectangle (12,5.5);
\filldraw[red,fill opacity=0.125] (0,.5) rectangle (4,1.5);
\filldraw[red,fill opacity=0.125] (4,2.5) rectangle (8,3.5);
\filldraw[red,fill opacity=0.125] (8,4.5) rectangle (12,5.5);

  \draw (-.15,-8pt) node {$-6$};
  \draw (.85,-8pt) node {$-5$};
  \draw (1.85,-8pt) node {$-4$};
  \draw (2.85,-8pt) node {$-3$};
  \draw (3.85,-8pt) node {$-2$};
  \draw (4.85,-8pt) node {$-1$};
  \draw (6,-8pt) node {$0$};
  \draw (7,-8pt) node {$1$};
  \draw (8,-8pt) node {$2$};
  \draw (9,-8pt) node {$3$};
  \draw (10,-8pt) node {$4$};
  \draw (11,-8pt) node {$5$};
  \draw (12,-8pt) node {$6$};

  \draw[thick,red] (2,.5) .. controls +(down:.6cm) and +(up:1cm) .. (6,0);
  \draw[thick,red] (2,1.5) .. controls +(up:.6cm) and +(down:1cm) .. (6,2.5);
  \draw[thick,red] (6,3.5) .. controls +(up:.75cm) and +(down:0.5cm) .. (10,4.5);
  \draw[thick,red] (10,5.5) .. controls +(up:.6cm) and +(down:.7cm) .. (6,6);
\draw[ultra thick,->] (0,6) -- (12,6);
\draw[ultra thick,->] (0,0) -- (12,0);

\end{tikzpicture}
\caption{The $T(3,4)$ cabling applied to the pattern in Fig. \ref{fig:pattern}.}

\label{fig:HWCabling}
\end{figure}

\newcommand{\m}{\iota}
\newcommand{\isoo}[1]{\vnp{#1}_{\cong}}
\newcommand{\car}{Hd}
\newcommand{\cdr}{Tl}
\newcommand{\Hd}{\car}
\newcommand{\Tl}{\cdr}

\section{Colored knot Floer homology}\label{altsec}

Theorem \ref{thm:stablelim} constructs a categorification of the
$S^r$-colored Alexander polynomial from Section \ref{ncoloredalexsec}. Corollary \ref{cor:structurecor} contains a
structure theorem for this filtered chain complex which will be used in
Section \ref{recsec}.

\newcommand{\subdiv}{subdiv}
\newcommand{\dilate}{\subdiv}

\begin{defn}{(subdivision)}\label{def:dilation}
Recall from Remark \ref{rmk:vertsimp} that by vertically-simplifying, every chain complex $C$ over a field such as $\FF_2$ can be decomposed into a sum
  $$C\cong  \left(\oplus_i \Lambda(x_i) \right) \oplus D,$$
where the differential of $C$ vanishes on $D$ and for each $\Lambda(x_i)$-term, $d(x_i) := 1_i$. If $C$ is bigraded then the \textit{length} of a $\Lambda(x_i)$-term is the Alexander or $q$-degree of $d$. With this in mind, we can introduce a chain complex  $\dilate(C)$ which fixes $D$ and replaces each $\Lambda(x_i)$-summand of length $A(d_i) = A(x_i) - A(1_i)$ by a direct sum of $A(d_i)$ terms of the form $\Lambda(x_{i,j})$ for $1\leq j \leq n_i$ each of length $1$. Alternatively, 
subdivision can be computed from the assignments below.
\begin{enumerate}
\item $\dilate(E\oplus F) := \dilate(E) \oplus \dilate(F)$
\item $\dilate(\FF_2) := \FF_2$
\item $\dilate((\Lambda(x_i),d)) := \oplus_{j=1}^{A(d)} (\Lambda(x_{i,j}),d_j)$ where $A(d_j) = 1$
\end{enumerate}  

If $(f_{x_i}, f_{1_i}) : \Lambda(x_i) \to \Lambda(x_i')$ are the components of a chain map $f$ then the assignment $\dilate(f_{x_i},f_{1_i}) := (f_{x_i},f_{1_i}, f_{x_i},f_{1_i},\ldots, f_{x_i},f_{1_i})$ extends $\dilate$ to a functor. This statement wont be used here, so further discussion is omitted.
  \end{defn}

The algebraic notion introduced above arises from an observation about the cabling procedure which is illustrated in Fig. \ref{fig:dilation} below.

\begin{figure}[!h]
\centering
\begin{tikzpicture}

\draw[thick, -,blue] (0,3) -- (3,3);

  \draw (1,3) node {$\bullet$};
  \draw (2,3) node {$\bullet$};

  \draw[thick,red] (0,2.5) .. controls +(right:.6cm) and +(left:.7cm) .. (1,3.5);
  \draw[thick,red,-] (1,3.5) -- (2,3.5);
  \draw[thick,red] (2,3.5) .. controls +(right:.7cm) and +(left:.6cm) .. (3,2.5);

\draw[very thick, ->] (3.5,3) -- (4, 3);

\draw[thick,blue, -] (4.5,3) -- (5,3.75) -- (5.5,3) -- (6,3.75) -- (6.5,3) -- (7, 3.75) -- (7.5,3);
  \draw (5,3.75) node {$\circ$};
  \draw (5.5,3) node {$\bullet$};
  \draw (6,3.75) node {$\circ$};
  \draw (6.5,3) node {$\bullet$};
  \draw (7,3.75) node {$\circ$};

  \draw[thick,red] (4.5,2.5) .. controls +(right:.6cm) and +(left:.7cm) .. (5.5,3.5);
  \draw[thick,red,-] (5.5,3.5) -- (6.5,3.5);
  \draw[thick,red] (6.5,3.5) .. controls +(right:.7cm) and +(left:.6cm) .. (7.5,2.5);

\end{tikzpicture}
\caption{Subdividing a bigon when $A(d)=2$. For $\circ$-dot notation, see \ref{circlenotation}.}
\label{fig:dilation}
\end{figure}
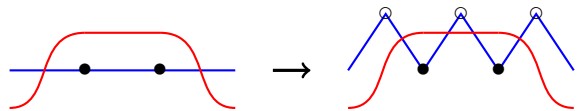

\begin{example}\label{example:dilate}
  Let $\Delta_m := \Lambda(\xi)$ where $d(\xi) = 1$ and $\vnp{\xi}_A = m$ be the $2$-term chain complex of Alexander degree $m$. Then the chain complex $D_m := \dilate(\Delta_m)$ consists of $m$, $A$-degree $1$, $2$-term complexes:
  $$D_m = \oplus_{i=1}^m \Lambda(\xi_i).$$

Moreover, the relation $D_m \cong D_{m-1} \oplus \Lambda(\xi_m)$ determines embeddings $\iota_m : D_{m-1} \hookrightarrow D_m$ and the limit of the sequence 
$$D_1 \xto{\iota_1} D_2 \xto{\iota_2} D_3 \to \cdots$$
can be identified with the unknot homology $\scfk(U)$ in Example \ref{exampleunknotcomplex} (up to grading considerations which will depend upon $r$).
  \end{example}

\newcommand{\arclength}{$TODO$ }

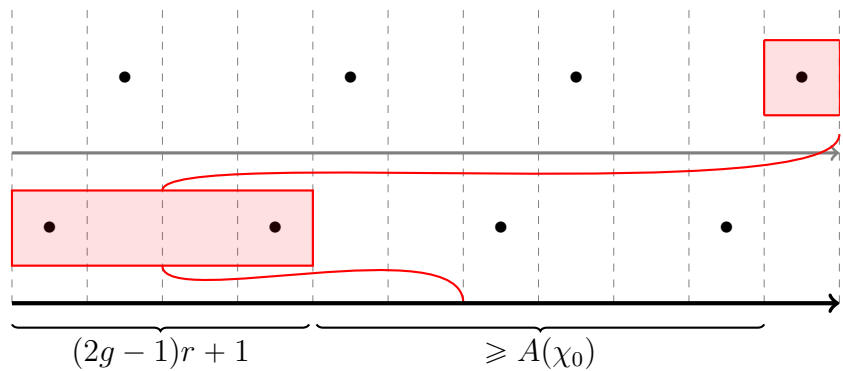
\begin{figure}[!h]
\centering
\begin{tikzpicture}
\draw[gray, very thin, dashed, step=1cm] (0,0) grid[ystep=0] (11,4);

\draw[very thick, gray, ->] (0,2) -- (11,2);

\foreach \x in {.5,3.5,6.5,9.5}
  \draw (\x,1) node {$\bullet$};

\foreach \x in {1.5,4.5,7.5,10.5}
  \draw (\x,3) node {$\bullet$};

\draw[thick,red] (10,2.5) -- (11,2.5);
\draw[thick,red] (10,3.5) -- (11,3.5);
\draw[thick,red] (10,2.5) -- (10, 3.5);

\draw[thick,red] (0,.5) rectangle (4,1.5);
\filldraw[red,fill opacity=0.125] (0,.5) rectangle (4,1.5);
\filldraw[red,fill opacity=0.125] (10,2.5) rectangle (11,3.5);

\draw [thick, decoration={brace,mirror}, decorate]  (0,-8pt) -- (4-0.05,-8pt) node [pos=0.5,anchor=south,yshift=-0.75cm]   {$(2g-1)r+1$};
\draw [thick, decoration={brace,mirror}, decorate]  (4.05,-8pt) -- (10,-8pt) node [pos=0.5,anchor=south,yshift=-0.75cm]   {$\geq A(\dgen)$};

  \draw[thick,red] (2,.5) .. controls +(down:.6cm) and +(up:1cm) .. (6,0);
  \draw[thick,red] (2,1.5) .. controls +(up:.6cm) and +(down:1cm) .. (11,2.25);
\draw[ultra thick,->] (0,0) -- (11,0);
\end{tikzpicture}
\caption{A schematic for Lemma \ref{lem:cablingstr}. The head consists of contributions from the lowest box. The tail begins on the non-compact curve at $\dgen$ and continues until reaching the penultimate box.}
\label{fig:cablingstr}
\end{figure}

\newcommand{\Leftbound}{\BC{???}}
\newcommand{\Rightbound}{\BC{???}}

\begin{defn}\label{onecoloreddef}
The $1$-colored knot Floer complex is defined to be the usual knot Floer complex, $S^1\widehat{CFK}(K) \defeq \hC(K)$, from Eqn. 
\eqref{eq:hatc}. 
Much like $r$-colored knot Floer complexes introduced later, the $1$-colored complex has a ``head-tail'' decomposition
$$S^1\widehat{CFK}(K) \cong \Hd(K) \opp \Tl(K)$$
compatible with the vertical-simplification Notation \ref{notation:vertsimp}, in the sense that
$$\Hd(K) := \opp_{i\in I} \Lambda(x_i) \conj{ and }  \Tl(K) := \FF_2\inp{\dgen}.$$
  \end{defn}

\begin{lemma}\label{lem:cablingstr}
Let $g$ be the genus of the knot $K$ and $K_{r,rn+1} \defeq C_{r,rn+1}(K)$ be the $(r,rn+1)$-cable of $K$. When $r>1$ and $n > 2g-1$  the chain complex 
$\hC(K_{r,rn+1})$ admits a ``head-tail'' decomposition, an ungraded isomorphism of the form:
\begin{align*}
\tau_A^{\leq rn+1}\hC(K_{r,rn+1}) & \cong \dilate(\Hd(K)) \oplus \subdiv(\Delta_m)
\end{align*}
where $m=n-g-\tau(K)-1$, $\tau_A^{\leq n}C := \{ x\in C : A(x) \leq n \}$, $\Hd(K)$ appears in Def. \ref{onecoloreddef} and $\Delta_m$ in Ex. \ref{example:dilate}.
\end{lemma}
\begin{proof}

The basic setting is as follows. The curve invariant $\og_{K_{r,rn+1}}$ in $\overline{T}^r_M$ associated to $(r,rn+1)$-cable $C_{r,rn+1}(K)$ of $K$ is obtained from the curve invariant $\og_K$ of $K$ using the algorithm described in Theorem \ref{thm:cablecurve}. 
By Theorem \ref{thm:IChfk}, 
$$\cfk(C_{r,rn+1}(K)) \simeq \hC(K_{r,rn+1}) = CF(\og_{K_{r,rn+1}},\omu)\ott_{\FF_2[U]} \FF_2.$$ 
implies that the knot Floer complex is the intersection Floer
complex.
In keeping with Rmk. \ref{rem:geodesic}, we choose a
length-minimizing representative of $\og_K$ and a length-minimizing representative of the corresponding cabling
invariant $\og_{K_{r,rn+1}}$. With these choices, the generators of  $\hC(K_{r,rn+1})$ form a reduced basis with
  respect to the vertical differential $\partial^{v}$. (Note by assuming a minimal representative of $\og_{K_{r,rn+1}}$ the boxes in the discussion of the cabling algorithm may not be so box-like, but the initial choice of a curve in a box suffices to identify and partition generators in the reduced complex.)

The key structural observation about the cabling algorithm (Thm. \ref{thm:cabling}) is that when the $(r,rn+1)$-cabling parameter $n > 2g-1$ is sufficiently large, two things are true:
\begin{enumerate}
\item the two copies of $\og_K$ in the penultimate and final rows occupy disjoint heights. 
\item the arc that connects these two copies of $\og_K$ is horizontal and its length is proportionate to $n$
\end{enumerate}
This observation is illustrated in Fig. \ref{fig:cablingstr}. The head is determined by the copy of $\og_K$ in the bottom row. The tail formed by the intersection points on the curve between the generator $\dgen$ and the start of the second copy of $\og_K$.

{\em Head identification:} 
Let $\og_{rK}$ be the ``$r$-scaled copy''
of $\og_K$ contained within the box in the bottom row of the cabling
construction. Since $\og_K$ is bound within the region of Alexander length $2g$, the curve $\og_{rK}$ is bound inside the lowest 
$(2g-1)r+1$ terms. Set
$$\hC_r(K) \defeq H_*(\textit{CF}(\og_{rK},\overline{\mu})).$$
Since the boxes are topologically identical we need only
analyze how changing $\omu$ to the zig-zag curve (as in Fig. \ref{fig:HWCabling}) changes the chain complex.
We assume that the chain complexes $\hC(K)$ and $\hC_r(K)$ have been vertically simplified and follow the conventions of \S\ref{sec:vertsimp}. 
Following Cor. \ref{cor:structurecor} (3), the vertical generator $\dgen$ is the distinguished vertical segment which connects to the box in the penultimate row. The remaining vertical segments correspond to terms $\Lambda(x_i)$ with $d(x_i)=1_i$ each of which corresponds to a bigon of length $n_i$.  To each such term there is a bigon $B$ appearing as the left-hand side of Fig. \ref{fig:dilation}. Associated to each such region is a subdivided region appearing as the right-handside of Fig. \ref{fig:dilation}.

Let us explain in detail how a bigon is subdivided. Suppose generators $x_i$ and $1_i$ from $\hC(K)$ are connected by a length $n_i$ bigon. All of the bigons are always $\og_{rK}$ on the right and $\omu^r$ on the left.
In $\hC(K_{r,rn+1})$, the arc connecting the top endpoints of their associated curve segments makes $2(n_i-1)$ additional intersections with $\overline{\mu}^r$. Introduce new generators $x_{i,j}$ and $1_{i,j}$. Set $x_{i,n_i} \defeq x_i$ and $1_{i,1} \defeq 1_i$ so that 
\begin{enumerate}
\item the bigon $B$ from $x_{i,j}$ to $1_{i,j}$, covers a single $z$-marked point ($n_z(B)=1$) for $1 \leq j \leq n_i$,
\item the bigon $B$ from $x_{i,j}$ to $1_{i,j+1}$, covers $r-1$ $w$-marked points ($n_w(B)=r-1$) for $1 \leq j \leq n_i -1$,
\end{enumerate}

with $\overline{T}^r$ understood to contain pairs of $z$ and $w$ marked points (as in the discussion following Theorem \ref{thm:ICpairing}). Labeled in this way, we collect these generators according to increasing Alexander grading by pairs admitting a bigon covering a single $z$ marked point:
\[
1_{i,1} \from x_{i,1}, 1_{i,2}\from x_{i,2}, \dots, 1_{i,j}\from  x_{i,j}, \dots, 1_{i,n_i}\from x_{i,n_i}
\]
Each pair $1_{i,j}\from x_{i,j}$ is connected by a bigon pictured on the left in Figure \ref{fig:commonbigons}. This collection agrees with $\subdiv(\Lambda(x_i))$ from Def. \ref{def:dilation}.

The above only specifies an embedding of $\dilate(\Hd(K)) \hookrightarrow \tau^{\leq rn+1}_A\hC(K_{r,rn+1})$. There can be no additional generators besides those accounting for the tail below because, after performing step (3) of the cabling algorithm, the pairs of curves $(\omu^r, \og_{rK})$ and $(\omu, \og_K)$ can be assumed to be equal except for the subdivision occuring due to the adding of basepoints.

{\em Tail identification:}
The remaining generators 
 are due to the arc connecting the distinguished generator $\dgen$ to the copy of $\og_{rK}$ in the penultimate row, the arc 
begins outside of the lowest box in Fig. \ref{fig:cablingstr} and continues until it reaches the start of the second highest box in Fig. \ref{fig:cablingstr}.
The $q$-grading of the distinguished generator is given by $A(\dgen) = r\tau(K)$, after this the arc forms pairs of intersections identical to those in the discussion of the head above until it reaches the height of the left edge of the bounding box in the the penultimate row. The particulars of how the arc connect to the second copy of $\og_K$ do not change our estimate. This collection of generators introduced by the arc can be identified with the dilation $\subdiv(\Delta_m)$ where $m=n-g-\tau(K)-1$.

  \end{proof}

In order to identify the limit of $(r,rn+1)$-cablings of $K$ as $n\to \infty$  in Theorem \ref{thm:stablelim}, we next recall Rozansky's convergence criteria.

Suppose that $\mathcal{\eA}$ is an additive category and $K^+(\eA)$ is the homotopy category of chain complexes over $\eA$ which are unbounded in positive degree. If $\m : A\to B$ is a map then the {\em isomorphism order} is defined to be
$$\isoo{\m} := \max \{ n : \tau^{\leq n}(\m) \textnormal{ is an isomorphism } \},$$
where $\tau^{\leq n} : K^+(\eA) \to K^{\leq n}(\eA)$ is truncation above degree $n$. 

\begin{lemma}{(\cite[Theorem 3.4]{roz})}\label{lem:limit}
  Suppose that $\mathcal{\eA}$ is an additive category and $K^+(\eA)$ is the associated homotopy category of chain complexes. Then a directed system has a limit 
if and only if the isomorphism orders diverge.
  $$\lim_{\to} \left(C_1 \xto{\m_1} C_2 \xto{\m_2} C_3 \xto{\m_3} \cdots\right) \conj{ $\Leftrightarrow$ }  \lim_{i\to \infty} \isoo{\m_i} = \infty$$
  \end{lemma}

Recall the element $\bgen\in\hC(K)$ from Notation \ref{notation:vertsimp} (4) is the element of lowest Alexander grading and least Maslov grading. For any knot $K$, the chain complex $\vnp{\bgen}^{-1}\cfk(K) = t^{-M(\bgen)} q^{-A(\bgen)}\hC(K)$ has been shifted so that the generator $\vnp{\bgen}^{-1}\bgen$ lies in $(t,q)$-degree $(0,0)$. The corollary uses the lemma to construct chain complexes $\scfk(K)$.

\begin{cor}\label{thm:stablelim}
For each knot $K$ in $S^3$, there is a limit of filtered chain complexes 
$$\scfk(K) := \lim_{n\to \infty} \vnp{b_{r,n}}^{-1} \hC(K_{r, rn+1})$$
where $K_{r,rn+1} = C_{r, rn+1}(K)$ is the $(r,rn+1)$-cable of $K$ and $b_{r,n}\in \hC(K_{r,rn+1})$ is the bottom generator.  The homology of this chain complex is an invariant of the knot $K$.
\label{thm:StableLimit}
\end{cor}

\begin{proof}
  By virtue of Lemma \ref{lem:cablingstr}, there are ungraded isomorphisms
$$\tau_A^{\leq rn+1}\hC(K_{r,rn+1})  \cong \dilate(\Hd(K)) \oplus \subdiv(\Delta_m)$$
  Setting $C_{r,n} \defeq \tau_A^{\leq rn+1}\vnp{b_{r,n}}^{-1}\hC(K_{r,rn+1})$.
Under these identifications we can define inclusion maps
  $\iota_n : C_{r,n} \hookrightarrow C_{r,n+1}.$
  The maps $\iota_n$ are identity on $\dilate(\Hd(K))$ and agree with the inclusions $\iota_n$ appearing in Ex. \ref{example:dilate} above  on tail terms. The maps $\iota_n$ can be upgraded to graded maps because the curves and associated grading data are identical in the regions under consideration. The isomorphism order diverges because the reduced complexes have the structure Eqn. \eqref{eq:vertsimp} and the maps $\iota_n$ are identity maps. We set
  $$\scfk(K) \defeq \lim_{n\to\infty} C_{r,n}$$

The homology of $S^r\cfk(K)$ is an invariant of $K$ because $\og_K$ and the absolute Maslov and Alexander gradings of the generators $b_{r,n}$ involved in the grading shifts above are invariants of $K$. Additionally, the graded Euler characteristic of $S^r$-colored knot Floer homology matches the $S^r$-colored Alexander polynomial of $K$ because the Euler characteristic commutes with limits (matching Definition \ref{coloredalexdef}).
\end{proof}

\begin{figure}[!h]

\centering

\begin{tikzpicture}[scale=1.1]

\draw[gray, very thin, dashed, step=1cm] (5,.5) grid[ystep=0] (9,2.5);
\draw[thick,blue] (5,1.5) -- (9,1.5);
  \draw (5,0.125) node {$-2$};
  \draw (6,0.125) node {$-1$};
  \draw (7,0.125) node {$0$};
  \draw (8,0.125) node {$1$};
  \draw (9,0.125) node {$2$};


  \draw[thick,red] (7,.5)
  .. controls +(up:.5cm) and +(down:0.3cm) .. (6,1.5)
  .. controls +(up:.3cm) and +(left:0.3cm) .. (6.5,2)
 .. controls +(right:.3cm) and +(left:0.3cm) .. (7.5,1)
 .. controls +(right:.3cm) and +(down:0.3cm) .. (8,1.5)
 .. controls +(up:.3cm) and +(down:.5cm) .. (7,2.5);

\filldraw[red,fill opacity=0.125] (6,1) rectangle (8,2);
\draw[thin,red] (6,1) rectangle (8,2);

\foreach \x in {5.5,6.5,7.5,8.5}
  \draw (\x,1.5) node {$\bullet$};
\draw[ultra thick,->] (5,2.5) -- (9,2.5);
\draw[ultra thick,->] (5,.5) -- (9,.5);

\draw (4.8,1.5) node {$\omu$};
\draw (7,2.75) node {$\og$};

  \draw (8, 1.5) node [mark=diamond] {};
 \draw (8.25, 1.5) node {$\dgen$};
  \draw (6-0.15, 1.5) node {$\bgen$};

  \end{tikzpicture}
\caption{The right-handed trefoil is $\og_K$ placed into a box with distinguished generators. Each $\bullet$ contains a $z$ $\epsilon$-above and a $w$ $\epsilon$-below.}

\begin{tikzpicture}
\draw[gray, very thin, dashed, step=1cm] (0,0) grid[ystep=0] (12,6);
\draw[very thick, gray, ->] (0,2) -- (12,2);
\draw[very thick, gray, ->] (0,4) -- (12,4);

\draw[thick,blue] (0,2) -- (.5,1) -- (1.5,3) -- (2.5,5) -- (3.5,1);
\draw[thick,blue] (3.5,1) -- (4.5,3) -- (5.5,5) -- (6.5,1);
\draw[thick,blue] (6.5,1) -- (7.5,3) -- (8.5,5) -- (9.5,1);
\draw[thick,blue] (9.5,1) -- (10.5,3) -- (11.5,5) -- (12,4);

\foreach \x in {.5,3.5,6.5,9.5}
  \draw (\x,1) node {$\bullet$};

\foreach \x in {1.5,4.5,7.5,10.5}
  \draw (\x,3) node {$\bullet$};

\foreach \x in {2.5,5.5,8.5,11.5}
  \draw (\x,5) node {$\bullet$};

\draw[thick,red] (0,.5) rectangle (4,1.5);
\filldraw[red,fill opacity=0.125] (0,.5) rectangle (4,1.5);

  \draw[thick,red] (10,.5)
  .. controls +(left:.5cm) and +(down:1cm) .. (0,1)
  .. controls +(up:.3cm) and +(left:0.3cm) .. (.5,1.4)
 .. controls +(right:.3cm) and +(left:1cm) .. (3.5,.6)
.. controls +(right:1cm) and +(left:1cm) .. (10,2.5);

  \draw (0,-8pt) node {$0$};
  \draw (1,-8pt) node {$1$};
  \draw (2,-8pt) node {$2$};
  \draw (3,-8pt) node {$3$};
  \draw (4,-8pt) node {$4$};
  \draw (5,-8pt) node {$5$};
  \draw (6,-8pt) node {$6$};
  \draw (7,-8pt) node {$7$};
  \draw (8,-8pt) node {$\cdots$};

  \draw[thick,red] (10,.5) -- (12,.5);
  \draw[thick,red] (10,2.5) -- (12,2.5);
\draw[ultra thick,->] (0,6) -- (12,6);
\draw[ultra thick,->] (0,0) -- (12,0);

  \draw (6.1, 1.6) node {$\dgen$};
  \draw (.2, 1.1) node {$\bgen$};

\end{tikzpicture}
\caption{The limit of the right-handed trefoil cabling process applied to a pattern Fig. \ref{fig:pattern} after removing extraneous bigons. In the first figure the generators satisfy $A(\dgen)-A(\bgen) = 2$, but here this difference $A'(\dgen)-A'(\bgen)=2*3$ is scaled by a factor of $r=3$.}

\label{fig:RScaling}
\end{figure}

\begin{proposition}\label{prop:grading}
There is an (ungraded) isomorphism of vector spaces
\[
\psi^r: S^r\cfk(K) \rightarrow S^{r+1}\cfk(K).
\]
The $(t,q)$-degrees of $\psi^r$ are affine linear functions of $r$.
\label{prop:cableincrement}
\end{proposition}

\begin{proof}
If we forget about gradings then there are canonical isomorphisms $\psi^r : \scfk(K) \to \sscfk(K)$ by virtue of the $r$-independence of the expression  from Cor. \ref{thm:stablelim}:
$$\scfk(K) = \subdiv(\Hd(K)) \opp \lim_{n\to\infty} \subdiv(\Delta_{n-g-\tau-1}).$$ 
Notice that it suffices to show affine
linearity of $\psi^r$ on generators $1_{i}$ since the
others are related to an even-index generator by a bigon containing a $z$
marked point which introduces a grading difference that is independent of
$r$. We will show affine linearity of $\psi^r$ in Alexander and Maslov gradings separately.

{\em Alexander gradings.}
For a generator $y$, we write $A(y)$ for the Alexander grading and $A'(y)$ for the Alexander grading in the shifted complex from Cor. \ref{thm:stablelim}.
We will first investigate grading changes associated to $\dgen$ and then
turn to the other generators. The Alexander grading of $\dgen$ is
straightforward to determine by looking at the $r$-scaled, shifted copy of
$\og_K$ in the final row of the cabling algorithm. The generator $\bgen$,
which we recall has least Maslov grading among the generators with minimal
Alexander grading, sits at height $0$. Being a generator with least
Alexander grading, its associated curve segment necessarily turns to the right, or upward in height, to
intersect $\overline{\mu}^r$ between the two bottom rows, as otherwise
there would be a generator with a smaller Alexander grading. Likewise, the
curve segment associated to $\dgen$ also turns upward to begin the tail
between the final two rows. This means that the difference in their
heights $A'(\dgen) - A'(\bgen)$ is divisible by $r$.
 (see Figure \ref{fig:RScaling}), and if $h$ measures heights in $\hC(K)$ we see that 
\begin{equation}\label{eq:Ashift1}
A'(\dgen)-A'(\bgen) = r(A(\dgen)-A(\bgen)).
\end{equation}
We know that the difference $A(\dgen)-A(\bgen)$ is precisely $\tau(K) + g(K)$ since the Alexander grading of $x$ is $\tau(K)$ by definition, and the Alexander grading is shifted before $r$-scaling by $-A(b) = g(K)$. Thus, the shifted Alexander grading of $\dgen$ in $S^r\cfk(K)$ is 
\begin{equation}\label{eq:AshiftF}
A'(\dgen) = r(\tau(K)+g(K)).
\end{equation}
Finally, any other generator $1_i$ coming from a curve segment that turns upward, (the complex is reduced), has its Alexander grading differ from that of $\dgen$ as a multiple of $r$ for the same reasons.

{\em Maslov gradings.}  Since the generator $\bgen$ can belong to any
component of the curve $\og_K$, determining Maslov gradings of generators
relative to that of $\bgen$ may involve traversing a grading arrow which
introduces a non-zero $2\text{Wght}(B)$ term in Hanselman's
Eqn. \eqref{eq:grformula}.  Fortunately, the weight $m$ of any such grading
arrow is independent of $r$\footnote{Specifically, we can always take grading arrows to lie below $\omu$, since no marked points cross the arrows as we move to $\og_{rK}$, none of its features change.},
since a grading arrow belonging to the bottom row does not
cross over any marked points when constructing the cabled curve
$\oga_{K_{r,rn+1}}$.

We know from the general algorithm to build $\oga_K$ from
$\textit{CFK}^-(K)$ (not necessarily using a horizontally- and vertically-simplified
basis) from Prop. \ref{prop:HFfromCFK} that a sequence of
immersed bigons connects any two generators in
$\vnp{b_{r,n}}^{-1}\hC(K_{r,rn+1})$. Let $\displaystyle \ast_i B_i$ be the 
concatenation of such a sequence of bigons which connects the bottom generator $\bgen$ to the distinguished generator $\dgen$ (of the copy of the curve $\og_K$ in the lowest row). Using
Eqn. \eqref{eq:grformula}, we have that
\begin{align}\label{eq:Mshift1}
M(\dgen)-M(\bgen) &= 2(\text{Wind}_w+\text{Wght}-\text{Rot})(\ast_{i} B_i) \\
&= 2\text{Wind}_w(\ast_i B_i) + 2m - 2\text{Rot}(\ast_{i} B_i).
\end{align} 
In this formula, $m$ is the weight of the grading arrow traversed by the first bigon
$B_1$ to go from $\bgen$ to a generator on the distinguished component $\og_0 \subset \og_K$ which ``wraps around the cylinder.''  If $\bgen$ already lies on $\og_0$ then $m=0$ because no grading arrow is required.

Observe that the next term, $\text{Rot}(\ast_i B_i)$, does not depend on $r$, since
the bigons between intersections in $\og_{K_{r+1, (r+1)n+1}} \pitchfork \omu^{r+1}$
can be assumed to belong to the two bottom rows
and are unchanged under a slight homotopy of $\overline{\mu}$ to
allow for an extra row above. 

However, the $\text{Wind}_w(\ast_i B_i)$-term will change because of the extra row of marked
points. Observe that any bigon $B_i$ between generators of the two bottom
rows contains either collections of the $(r-1)$-many $w$ marked points
from the rows above or collections of single $w$
marked points from the bottom row. This implies that $\text{Wind}_w(B_i)$
is a linear combination of $(r-1)$ and $1$, so that we may sum these
contributions and upgrade Equation \ref{eq:Mshift1} to be
\begin{align}\label{eq:MshiftF}
M(x)-M(b) &= 2\text{Wind}_w(\ast_i B_i) + 2m - 2\text{Rot}(\ast_{i} B_i) \\
&= 2\alpha (r-1) + 2\beta + 2m - 2\text{Rot}(\ast_i B_i) \\
&= 2\alpha (r-1) + \delta,
\end{align}
with $\delta$ capturing the $r$-invariant contribution to the grading difference. An affine linear function of $r-1$.

Since there is a sequence of bigons connecting the distinguished generator
$\dgen$ to any generator $y$ in the image of the vertical differential (such as
$1_{i}$), the same reasoning about the location and shape
of bigons, as well as the weight of a potential grading arrow, applies to
show that
\begin{equation}\label{eq:MshiftExtra}
M(y)-M(\dgen) = 2\kappa (r-1) + \nu.
\end{equation}
(Note that an affine function of $(r-1)$ is an affine function of $r$.)
\end{proof}

\begin{example}\label{exampleunknotcomplex}
The illustration below shows two immersed curves: the stable curve
$[U]_r$ associated to the unknot is drawn in red and the lift of the meridian
$\mu$ is the blue curve which zigzags between the two different types of base
points (see Not. \ref{circlenotation}). The points of intersection between the two curves constitute
generators for the chain complex $\scfk(U)$, see
Fig. \ref{fig:StableUnknot}.

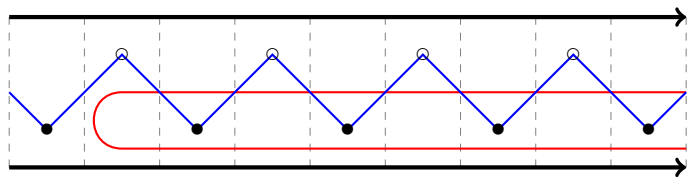
\begin{figure}[h!]
\centering
\begin{tikzpicture}
\draw[gray, very thin, dashed, step=1cm] (0,.5) grid[ystep=0] (9,2.5);
\draw[ultra thick,->] (0,2.5) -- (9,2.5);
\draw[ultra thick,->] (0,.5) -- (9,.5);

\draw[thick, red] (9,.75) -- (1.5,.75);
\draw[thick, red] (9,1.5) -- (1.5,1.5);
\draw[thick, red] (1.5,.75) .. controls (1,.75) and (1,1.5) .. (1.5,1.5);

\draw[thick,blue] (0,1.5) -- (.5,1) -- (1.5,2) -- (2.5,1) -- (3.5,2) -- (4.5,1) 
-- (5.5,2) -- (6.5,1) -- (7.5,2) -- (8.5,1) -- (9,1.5);
 \foreach \x in {0.5,2.5,4.5,6.5,8.5}
     \draw (\x,1) node {$\bullet$};

 \foreach \x in {1.5,3.5,5.5,7.5}
     \draw (\x,2) node {$\circ$};

  \end{tikzpicture}
\caption{The stable curve associated to the unknot.}
\label{fig:StableUnknot}
\end{figure}

The Poincar\'{e} polynomial of the $S^r$-colored unknot $\fcfk(U)$ is
$$\pnp{U}_{S^r}(t,q) = \frac{1+tq}{1-t^{2(r-1)} q}.$$
As a vector space 
\begin{align*}
  \fcfk(U) &\cong \FF_2[u] \ott \Lambda(\xi)\\
           &\cong \FF_2\inp{u^n\xi^k : n\in \ZZg, k\in \{0,1\} }.
\end{align*}
The grading is given by $\vnp{u} = t^{2(r-1)} q^r$ and $\vnp{\xi} = t^1q^1$ so that $\vnp{u^n\xi^k}=\vnp{u}^n\vnp{\xi}^k$. There is a differential $d : \fcfk(U) \to \fcfk(U)$ which is determined by the assignments
$$d(u):=0 \conj{ and } d(\xi) := 1$$
after extending by the Leibniz rule. The filtration is given by
$$F^\ell(\fcfk(U)) := \{ x : \vnp{x}_q \leq \ell \}.$$
Notice that, since the differential $d$ decreases $q$-degree by one, the differential of the associated graded chain complex
$\gscfk(U) = \gr(\fcfk(U))$  is identically zero. This is a general phenomenon.
  \end{example}

\begin{example}\label{exampletrefoilcomplex}
In a similar way, there is a filtered chain complex $\fcfk(3_1)$ associated to the right-handed trefoil, see Fig. \ref{fig:RScaling}. The stable curve is pictured in here.

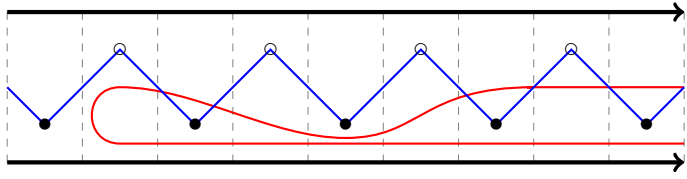
\begin{figure}[h!]
\centering
\begin{tikzpicture}
\draw[gray, very thin, dashed, step=1cm] (0,.5) grid[ystep=0] (9,2.5);
\draw[ultra thick,->] (0,2.5) -- (9,2.5);
\draw[ultra thick,->] (0,.5) -- (9,.5);

 \draw[thick,red] (9,1.5)
 .. controls +(left:1cm) and +(left:.5cm) .. (7,1.5)
 .. controls +(left:1.5cm) and +(right:1cm) .. (4.5,.825)
 .. controls +(left:1cm) and +(right:1cm) .. (1.5,1.5);

 \draw[thick, red] (9,.75) -- (1.5,.75);
\draw[thick, red] (1.5,.75) .. controls (1,.75) and (1,1.5) .. (1.5,1.5);

\draw[thick,blue] (0,1.5)--(0.5,1) -- (1.5,2) -- (2.5,1) -- (3.5,2) -- (4.5,1) 
-- (5.5,2) -- (6.5,1) -- (7.5,2) -- (8.5,1) -- (9,1.5);
 \foreach \x in {0.5,2.5,4.5,6.5,8.5}
     \draw (\x,1) node {$\bullet$};

 \foreach \x in {1.5,3.5,5.5,7.5}
     \draw (\x,2) node {$\circ$};

  \end{tikzpicture}
\caption{Stable curve associated to the right-handed trefoil}
\label{fig:RaStableRHT}
\end{figure}

The Poincar\'{e} polynomial of the right-handed trefoil is given by two terms associated to the intersection points near the ``head'' of the curve, together with a ``tail'' which can be identified with a shifted copy of the $S^r$-colored unknot found in Example \ref{exampleunknotcomplex}.
$$\pnp{3_1}_{S^r}(t,q) = 1 + tq + t^{4r-2}q^{4r} \pnp{U}_{S^r}$$
As a vector space
\begin{align*}
  \fcfk(3_1) &:= \Lambda(c) \oplus t^{4r-2}q^{4r}\fcfk(U)\\
  &\cong \Lambda(c) \oplus t^{4r-2}q^{4r}\left(\FF_2[u]\otimes \Lambda(\xi) \right)
\end{align*}  
where $\Lambda(c) := \FF_2\inp{1_c, c}$ and $\Lambda(\xi) := \FF_2\inp{\xi, 1}$ and the bidegrees are given by $\vnp{c}=t^1q^1$, $\vnp{u} = t^{4r-2}q^{4r} t^{2(r-1)}q = t^{6r-1} q^{4r+1}$ and $\vnp{\xi} = t^{4r-2}q^{4r} t^1 q^1 = t^{4r-1} q^{4r+1}$. The differential $d : \fcfk(3_1)\to \fcfk(3_1)$ is determined by the assignements $d(c) := 1_c$, $d(u):=0$ and $d(\xi):=1$. 

As for the unknot, the filtration is defined by $q$-degree and the differential vanishes on the associated graded complex.
  \end{example}

The proposition below summarizes the information contained in this
section. The filtered graded chain complexes we constructed in this section
are similar in form to the right-handed trefoil above
Ex. \ref{exampletrefoilcomplex} but with some extra $\Lambda(a)$-type
terms. The contents of this section which are used in \S\ref{recsec} are summarized by the proposition below.

\begin{prop}\label{propprop}
  For each knot $K$, the filtered chain complex $\fcfk(K)$ admits a canonical decomposition. In the category $H^0(F^*\Ch_\ZZ)$ there are isomorphisms
\begin{equation}\label{decompeqn}
  \fcfk(K) \cong \left(\Lambda(a_{1,r})\opp\Lambda(a_{2,r})\opp\cdots\opp\Lambda(a_{n,r})\right)\opp \theta^r \fcfk(U).
  \end{equation}
where $\theta$ is an overall $(t,q)$-degree shift.

The {\em head} of the decomposition consists of a finite sum $\opp_i \Lambda(a_{i,r})$ of 2-dimensional complexes,
\begin{equation}\label{extalgeq}
  \Lambda(a_{i,r}) \cong \FF_2\inp{1_{i,r},a_{i,r}}
\end{equation}  
The differential is defined to be $d(a_{i,r}) := 1_{i,r}$ with grading equal to $\vnp{a_{i,r}} = tq\vnp{1_{i,r}}$
Moreover, for each $1\leq i\leq n$, the grading $\vnp{1_{i,r}} = t^{mr+b} q^{nr+c}$ is an affine linear function in $r$.

On the other hand, the {\em tail} of the decomposition is a shifted copy of the unknot complex $\theta^r\fcfk(U)$ from Ex. \ref{exampleunknotcomplex},
  $$\fcfk(U) \cong \FF_2[u]\opp \Lambda(\xi)$$
the grading is determined by the assignments $\vnp{u} = t^{2(r-1)} q^r$ and $\vnp{\xi} = tq$ in the sense that $\vnp{u^n\xi^k}=\vnp{u}^n\vnp{\xi}^k$ and the differential is fixed by setting $d(u):=0$ and $d(\xi) := 1$.
  \end{prop}

\begin{remark}
The stable curve in the infinite annulus isn't an object in the same Fukaya
category as the curves associated to the cablings. We do not address here the
question of how to describe the category in which this curve lives.
  \end{remark}

\section{Weyl algebra and homological holonomicity}\label{wahhsec}
\newcommand{\xM}{\aM}
\newcommand{\xL}{\aL}
\newcommand{\xA}{\aA}

\subsection{Some algebraic definitions}\label{app1sec}

If $\eA$ is an abelian category then there is an additive category $F^*\eA$ of {\em filtered objects} in $\eA$. An object of the category $F^*\eA$ is a sequence $E := \{F^n E\}_{n\in \ZZ}$ with inclusions $F^{n+1} E\subseteq F^n E$. Maps $f : E \to E'$ are maps in $\eA$ which commute with inclusions. Similarly, if $\eA$ is an additive category then there is an additive category $\eA_{\ZZ}$ of {\em graded objects} in $\eA$. An object of $\eA_\ZZ$ is a sequence $\{E^n\}_{n\in\ZZ}$ of objects in $\eA$ and maps $f : E\to E'$ are collections $\{ f^n : E^n \to E'^{n} \}_{n\in\ZZ}$ of maps in $\eA$. When $\eA$ is abelian, there is a functor called {\em associated graded }
$$\gr : F^*\eA \to \eA_\ZZ$$ which is defined on objects by $\gr(\{F^n E\}_{n\in\ZZ})^m := F^m E/F^{m+1} E$.

For an additive category $\eA$, there is a category $\Ch(\eA)$ of chain complexes over $\eA$.

\begin{defn}{($\Ch(\eA)$)}\label{chdef}
A {\em chain complex} $E\in \Ob(\Ch(\eA))$ of consists of pair $(\{E^k\}_{k\in \ZZ}, d_E^k : E^k \to E^{k+1})$  where $E^k\in \Ob(\eA)$ and $d^{k+1}_Ed^k_E = 0$  for all $k\in \ZZ$. 

A {\em map} $f\in \Hom^k(E,E')$ of degree $k$ is a collection $f = \{ f^i : E^{i} \to E'^{k+i} \}_{i\in\ZZ}$ of maps in $\eA$. For any such map, there is a map $\d^k f \in \Hom^{k+1}(E,E')$ with components
$$(\d^k f)^i := d_{E'}^{k+i} f^i + (-1)^{i+1} f^{i+1}d_E^i$$
In this way, the set of maps $\Hom(E,E') = (\{\Hom^k(E,E')\}_{k\in\ZZ}, d^k : \Hom^k(E,E')\to \Hom^{k+1}(E,E'))$ is a chain complex. 
\end{defn}

The composition satisfies the Leibniz rule with respect to the differential making $\Ch(\eA)$ a category which is enriched over itself. More generally, a category which is enriched over chain complexes is called a {\em differential graded (dg) category}.

A dg category $\eC$ has a number of auxilliary structures stemming from the differential. A map $f$ is a {\em cycle} if $\d f=0$. Two maps, $f$ and $g$ are {\em homotopic} or $f\simeq g$ when there is a map $H$ such that $\d H = f-g$. There is a category $Z^0(\eC)$ consisting of degree zero cycles and the {\em homotopy category} $H^0(\eC) := Z^0(\eC)/\simeq$ consisting of cycles modulo boundaries. Two objects $E,F\in \Ob(\eC)$ are {\em homotopy equivalent} $E\simeq F$ when they are isomorphic $E\cong F$ in the homotopy category $H^0(\eC)$. An object $E\in\Ob(\eC)$ is {\em contractible} when $E\simeq 0$.

\subsection{Action of Weyl algebra}\label{weylsec}

The dg category $\SSeq$ introduced below will be our replacement for the collection of sequences $\Seq$.

\begin{defn}{($\SSeq(\eA)$)}
Suppose $\eA$ is a dg category. Then the dg category of sequences is defined by $\SSeq(\eA) := \prod_{i\in\ZZg} \eA$. In more detail, the objects are given by sequences $E=(E_i)_{i\in\ZZg}$ of objects $E_i\in \Ob(\eA)$. A map of degree $k$ from $E=(E_i)_{i\in\ZZg}$ to $E'=(E'_i)_{i\in\ZZg}$ is given by an element of the product
$$\Hom_{\SSeq(\eA)}^k(E,E') = \prod_{i\in\ZZg} \Hom_\eA^k(E_i,E'_i).$$
If $f : E\to E'$ then we write $f = (f_i)_{i\in \ZZg}$, or just $f=(f_i)$, for the sequence of maps which determine $f$.
Given two maps $f : E\to E'$ and $g : E'\to E''$, the composition and differential are defined componentwise:
\begin{equation}\label{compeq}
  (g\circ f) = (g_i\circ f_i)_{i\in \ZZg} \conj{ and } \d(f) := (d_{E_i}f_i)_{i\in\ZZg}.
  \end{equation}
The identity map is $1_E=(1_{E_i})_{i\in\ZZg}$. 
\end{defn}

For a dg category $\eA$, the Grothendieck group $K_0(\eA)$ is simply an abelian group. When a category $\eA$ has additional structure, this structure is reflected in $K_0(\eA)$.
For instance, if $\eA$ supports a $\ZZ$-action then the Grothendieck group $K_0(\eA)$ is a $\ZZ[q,q^{-1}]$-module. If $\eA = (\eA,\otimes,1)$ is a monoidal category then $K_0(\eA)$ is a ring. 
More details and additional discussion of the ideas surrounding the statement of the proposition below can be found in the reference \cite{KMS}.

\begin{prop}\label{weylprop}
  Suppose that $\eA$ is a dg category which supports a $\ZZ$-action. Then there are dg functors 
$$\xM, \xL : \SSeq(\eA) \to \SSeq(\eA)$$ 
and a natural isomorphism
\begin{equation}\label{Req}
  R : \xM\circ\xL \xto{\sim} t^2q \circ \xL \circ \xM.
  \end{equation}
which lifts the action of the Weyl algebra on the right-hand side of the map
$$\Phi : K_0(\SSeq(\eA))\to K_0(\eA)^{\ZZg} \conj{ given by } \Phi((E_i)_{i\in\ZZg}) := ([E_i])_{i\in\ZZg}$$
\end{prop}

\begin{proof}
In order to introduce the map $\Phi$, one uses $K_0(\eA\times \eB)\cong K_0(\eA)\times K_0(\eB)$, see \cite[Ex. 2.1.4, p. 75]{Kbook}.

Recall from \S\ref{qholalexsec} that there are maps $M : \Seq\to \Seq$ and $L : \Seq\to \Seq$ which are defined on functions $f$ by Eqn. \eqref{lmeq}. In the proof of the statement above we will construct functors $\xM$ and $\xL$ which have the same effect as $M$ and $L$ on objects. Then we will find the isomorphism $R$.

The functor $\xM : \SSeq(\eA)\to \SSeq(\eA)$ is given by $\xM := \prod_{i\in\ZZg} t^{2i-2}q^i$. In more detail, $\xM$ is defined on objects $E=(E_i)_{i\in \ZZg}$ by
$$\xM(E) := ((\xM E)_i)_{i\in \ZZg} \conj { where } (\xM E)_i := t^{2i-2}q^i E_i.$$
If $f : E\to E'$ is a map $f=(f_i)_{i\in\ZZg}$ of degree $k$ then there is a map $\xM(f) : \xM(E) \to \xM(E')$ of degree $k$, whose components are $(\xM f)_i := t^{2i-2} q^i f_i$. These assignments determine a dg functor by virtue of universal properties and being a componentwise dg functor. Alternatively, a short proof of the three properties below is included.
\begin{enumerate}[label=(\roman*)]
\item $\xM(g\circ f) = \xM(g)\circ \xM(f)$
\item $\xM(1_E) = 1_{\xM(E)}$
\item $d\xM(f)= \xM(df)$
\end{enumerate}
In each case, we use the functoriality of the assignment $f_i \mapsto  t^{2i-2} q^i f_i$.
For (i), $\xM(g\circ f)_i = t^{2i-2} q^i (g_i \circ f_i) = (t^{2i-2} q^i g_i) \circ (t^{2i-2} q^i f_i)= \xM(g)_i \circ \xM(f)_i$. For (ii), $\xM(1_{E})_i = t^{2i-2} q^i 1_{E_i} = 1_{t^{2i-2} q^i E_i} = 1_{\xM(E_i)}$. For (iii), $d\xM(f)_i = d(t^{2i-2}q^i f_i) = \xM(df)_i$.

On the other hand, the functor $\xL : \SSeq(\eA)\to \SSeq(\eA)$ is given by shifting objects and maps by one. On objects $E=(E_i)_{i\in \ZZg}$ the functor is defined by setting
$$\xL(E) := ((\xL E)_i)_{i\in \ZZg} \conj { where } (\xL E)_i := E_{i+1}$$
and if $f = (f_i)_{i\in \ZZg}$ is a map then $\xL(f)_i := f_{i+1}$. For properties (i), (ii) and (iii) above the proofs go as follows. For (i): $\xL(g\circ f)_i = (g\circ f)_{i+1} = g_{i+1}\circ f_{i+1} = \xL(g)_i \circ \xL(f)_i$. For (ii): $\xL(1_E)_i = 1_{E_{i+1}} = 1_{\xL(E)_i}$. For (iii): $d\xL(f)_i = df_{i+1} = \xL(df)_i$.

Finally, we compute the map $R$ in Eqn. \eqref{Req}. Given $f : E\to E'$, $f = (f_i)_{i\in\ZZ}$, the left-hand side is $(\xM\circ\xL)(f)_i = \xM(f_{i+1}) = t^{2i}q^{i+1} f_{i+1}$. On the right-hand side, we compute
$(\xL\circ\xM)(f)_i = t^{2i-2}q^{i}f_{i+1}$. The two sides differ by $t^2q$.
  \end{proof}

\begin{remark}
Computations suggest that in some cases it might be more natural to use $\xM:=\prod_{i} t^{2i}q^i$ or in Ex. \ref{pascalex} below, to use a two independent functors $\xM_{t^2} :=\prod_i t^{2(i-2)}$ and $\xM_q := \prod_i q^i$.
  \end{remark}

Let us now specialize our discussion to the dg category of $\eA = F^*\Ch_{\ZZ}$ (filtered) graded chain complexes. This category will be denoted by 
\begin{equation}\label{sseqdef}
\SSeq := \SSeq(F^*\Ch_\ZZ).
\end{equation}
Shifting the internal grading, together with the filtration degree, up by $1$ determines a functor $q : \SSeq\to \SSeq$ which
generates a $\ZZ$-action $\ZZ\cong\inp{q}$. Shifting the homological degree in $F^*\Ch_{\ZZ}$ yields a second automorphism $t : \SSeq\to \SSeq$ which commutes with the functors $\xM$ and $\xL$. In this way, the category $\SSeq$ is a module over the $t$-extended Weyl algebra 
\begin{equation}\label{texeq}
  \xA^+ := \xA\ott_{\ZZ}\ZZ[t,t^{-1}]
  \end{equation}
The homotopy category $H^0(\SSeq)$ is triangulated, so that the category $\SSeq$ is considered to be pre-triangulated \cite{BK}.

We next introduce a form of holonomicity, which we use in \S\ref{recsec}, as a general sort-of Whitehead tower.

\begin{defn}{(Holonomicity)}\label{homqholdef}
An object $E$ of a pre-triangulated category $\eC$ can be {\em assembled from a subcategory} $\eD\subset \eC$ if there is a length $n$ sequence of distinguished triangles in $H^0(\eC)$ of the form
\[
\begin{array}{cc}
X_1 \to X_0 \to Y_0 \to X_1[1]\\
X_2 \to X_1 \to Y_1 \to X_2[1]\\
\vdots\\
X_{n+1} \to X_n \to Y_n \to X_{n+1}[1].
\end{array}
\]
with $X_0 \cong E$, $X_{n+1} \cong 0$ and $\{Y_i\}_{i=0}^n\subset\Ob(\eD)$. Such a sequence of distinguished triangles is {\em non-trivial} if $Y_0\not\cong 0$.
If an algebra $B$ acts on a pre-triangulated category $\eC$ then let $\inp{B\cdot E} \subset \eC$ be the smallest thick pre-triangulated subcategory containing objects of the form $b(E)$ for $b\in B$. An object $E\in\Ob(\SSeq)$ is {\em holonomic} if $E\cong 0$ or $E\not\cong 0$ and it can be assembled in a non-trivial way from the smallest thick pre-triangulated category $\inp{\xA^+ \cdot E}$ containing orbit of the Weyl action on $E$.
  \end{defn}

Under the identification given by the map $\Phi$ from Prop. \ref{weylprop} this condition becomes $q$-holonomicity in the usual sense.

\begin{remark}
Our focus here is principally holonomicity as it is seen in relation to
conjectures in quantum topology, but in this ``categorified setting'' it
might also be interesting to investigate holonomicity properties of maps. For instance,
those assigned to cobordisms.
  \end{remark}

\subsection{Examples of holonomicity}\label{examplessec}

This section is concluded with a collection of examples related to the
definition of holonomicity from Def. \ref{homqholdef} above.

\begin{example}{(Constant sequence)}\label{constseqex}
If $C\in\Ob(F^*\Ch_\ZZ)$ then there is an object $\ul{C}\in \Ob(\SSeq)$ given by $\ul{C}_n := C$ for all $n\geq 0$. There is a map $f : \ul{C} \to \xL\ul{C}$ given by $f = (1_C : \ul{C}_n \to \ul{C}_{n+1})$. There is one distinguished triangle
  $$X_1 \to X_0 \xto{f} Y_0 \to X_1[1]$$
with $X_1 \cong 0$, $X_0 = \ul{C}$ and $Y_0 = \xL\ul{C}$.
\end{example}

\begin{example}{(Summing Sequence)}\label{summingseqex}
For each complex $C\in\Ob(F^*\Ch_\ZZ)$, there is an object $C^\opp \in \Ob(\SSeq)$ where $C^\opp_n := \opp_{i=0}^n C$ for $n\geq 0$. There is a map $i : C^\opp \to \xL(C^\opp)$, $i=(i_n)_{n\geq 0}$ where $i_n : \opp_{j=0}^n C\to \left(\opp_{j=0}^{n} C\right)\opp C$ is the inclusion of the first $n$-terms. Then the cone $\Cone(i)_n$ is homotopy equivalent to a single copy of $C$. So the cone on $f$ is isomorphic to the constant sequence, $\Cone(f)\cong \ul{C}$ in $H^0(\SSeq)$. Combining this observation with the construction for the constant sequence in Example \ref{constseqex} gives the following collection of distinguished triangles.
\[
\begin{array}{cc}
X_1 \to X_0 \xto{i} Y_0 \to X_1[1]\\
X_2 \to X_1 \xto{f} Y_1 \to X_2[1]\\
\end{array}
\]
where $(X_0,X_1,X_2)=(C^{\opp}, \ul{C},0)$ and $(Y_0,Y_1)=(\xL(C^{\opp}), \xL(X_1))$.

On the other hand, unwinding the sequence of extensions above gives the diagram below. 
\begin{center}
\begin{tikzpicture}[scale=10, node distance=2.5cm]
\node (A) {$C^{\opp}$};
\node (B) [right of=A] {$\xL(C^\opp)$};
\node (C) [below of=A] {$\xL(C^\opp)$};
\node (D) [below of=B] {$\xL\!\xL(C^\opp)$};
\draw[->] (A) to node {$i$} (B);
\draw[->] (C) to node {$\xL(i)$} (D);
\draw[->] (A) to node [swap] {$f$}  (C);
\draw[->] (B) to node {$\xL(f)$} (D);
\end{tikzpicture} 
\end{center}
Holonomicity could also be formulated in terms of contractible twisted complexes of this sort.
  \end{example}

The next example was explored by the first author and O. Yacobi circa 2012. This is the original motivation for Def. \ref{homqholdef}.

\begin{example}{(Pascal's Triangle)}\label{pascalex}
The Grassmannian $Gr(d,n)$ of $d$-planes in $\CC^n$ is stratified by Schubert varieties 
$Gr(d,n) = \cup_I X(I)$ \cite{Brion} 
where $I := (1\leq i_1 < i_2 < \cdots < i_d \leq n)$ and $\dim X(I) = 2\sum_{k=1}^d (i_k-k)$. If $H(d,n) := H_*(Gr(d,n);\ZZ)$ then there is an identification $H(d,n) = \opp_I \ZZ\inp{[X(I)]}$. The Poincar\'{e} polynomial of the homology can be identified with the $q$-binomial coefficients as follows,
$$\pnp{H(d,n)}(t,q^{\frac{1}{2}}) = \bigoplus_k q^{k} \dim H_{2k}(Gr(d,n)) = \left[\begin{array}{c} n\\ d \end{array}\right] = \frac{[n]!}{[d]![n-d]!}$$
where $[n]=(q^n-q^{-n})/(q-q^{-1})$ and $[n]!=[n][n-1]\cdots[1]$. This is a $q$-analogue of the binomial coefficient in the sense that the limit $q\to 1$ gives $n!/(d!(n-d)!)$.

There are degree zero maps $\a : q^{2(n-d)} H(d-1,n-1) \to H(d,n)$ and $\beta : H(d,n) \to H(d,n-1)$ which are defined by
\begin{align*}
   \a([X(i_1, \cdots, i_{d-1})]) := [X(i_1, i_2,\cdots, i_{d-1},n)]\conj{ and }\\ 
\beta([X(i_1, \cdots, i_d)]) := \left\{\begin{array}{ll} 0 &\textnormal{ for } i_d = n\\ 
                                      {[}X(i_1,i_2,\cdots,i_d){]} &\textnormal{ for } i_d<n 
                                      \end{array} \right.
\end{align*}
It follows from the definitions above that there is a short exact sequence,
$$0 \to q^{2(n-d)} H(d-1,n-1) \xto{\a} H(d,n) \xto{\beta} H(d,n-1) \to 0.$$
This short exact sequence, which determines Pascal's triangle, also produces many combinatorial examples in our setting.
\end{example}

\newcommand{\sru}{S^r(U)}
\begin{example}{(Unknot)}
Let's make an abbreviation by setting $\sru := \fcfk(U)$.
  From Ex. \ref{exampleunknotcomplex}, the unknot homology can be identified with
\begin{align*}
  \sru &\cong \FF_2[u] \ott \Lambda(\xi)\\
           &\cong \FF_2\inp{u^n\xi^k : n\in \ZZg, k\in \{0,1\} }.
\end{align*}
The grading is given by $\vnp{u} = t^{2(r-1)} q^r$ and $\vnp{\xi} = t^1q^1$ so that $\vnp{u^n\xi^k}=\vnp{u}^n\vnp{\xi}^k$. The differential is determined by setting $d(u):=0$  and $d(\xi) := 1$. 

To build a recurrence relation for the sequence $(S^r(U))_{r\in\ZZg}\in \SSeq$, we begin by observing that there are maps 
$$h_r : \sru\to \vnp{u} \sru \conj{ where } h_r(u^n\xi^k) := u^{n-1} \xi^{k}$$
the degree of $h_r$ is $\vnp{h_r} = \vnp{u}\vnp{u^{n-1}\xi^k}-\vnp{u^n\xi^k}=0$, and $h_r$ is a chain map which respects the $q$-filtration because it decreases $q$-degree: $h_r : F^\ell(\sru) \to F^{\ell-r}(\sru)\subset F^\ell(\sru)$ where $F^\ell(\sru) :=\{x : \vnp{x}_q\leq \ell\}$. The cone $\Cone(h_r)$ is homotopy equivalent to the 2-dimensional subcomplex $\Lambda(\xi):=\FF_2\inp{1,\xi}$. So there is a map $h : \sru \to \xM\sru$ given by $h:=(h_r)_{r\geq 0}$ such that
$$\Cone(h) \cong \ul{\Lambda(\xi)} \conj{ in } \SSeq$$
is isomorphic to the constant sequence for $\Lambda(\xi)$ in $\SSeq$. The recurrence relation is now completed by following Example \ref{constseqex} above, there is a map $f : \Cone(h)\to \xL(\Cone(h))$ so that $\Cone(f) \cong 0$ in $\SSeq$. 

The diagram below illustrates the recursion
\begin{center}
\begin{tikzpicture}[scale=10, node distance=2.5cm]
\node (A) {$\sru$};
\node (B) [right of=A] {$\xM\sru$};
\node (C) [below of=A] {$\xL\sru$};
\node (D) [below of=B] {$\xL\!\xM\sru$};
\draw[->] (A) to node {$h$} (B);
\draw[->] (C) to node {$\xL(h)$} (D);
\draw[->] (A) to node [swap] {$f$}  (C);
\draw[->] (B) to node {$f$} (D);
\end{tikzpicture} 
\end{center}
This diagram categorifies Lemma \ref{unknotlem}.
  \end{example}

The next example is more representative of the general case in \S\ref{recsec}.

\newcommand{\srt}{S^r(3_1)}
\begin{example}{(Right-handed trefoil)}
Set $\srt := \fcfk(3_1)$. 
From Ex. \ref{exampletrefoilcomplex}
\begin{align*}
  \srt &:= \Lambda(c) \oplus t^{4r-2}q^{4r}\fcfk(U)\\
  &\cong \Lambda(c) \oplus t^{4r-2}q^{4r}\left(\FF_2[u]\otimes \Lambda(\xi) \right)
\end{align*}  
This time the bidegrees are given by $\vnp{c}=t^1q^1$, $\vnp{u} = t^{4r-2}q^{4r} t^{2(r-1)}q = t^{6r-4} q^{4r+1}$ and $\vnp{\xi} = t^{4r-2}q^{4r} t^1 q^1 = t^{4r-1} q^{4r+1}$. The differential is given by $d(c) := 1_c$, $d(u):=0$ and $d(\xi):=1$. 

As with the unknot in Ex. \ref{exampleunknotcomplex} above, we define a map $f=(f_r)$, with components $f_r : \srt \to \vnp{u}\srt$ for which  $f_r(c):=0$, $f_r(1_c):=0$ and $f_r(u^n\xi^k):= u^{n-1}\xi^k$, but now the neglected term $\Lambda(c)$ doubles in the sense that
$$\Cone(f_r) \simeq \Lambda(c)\opp \vnp{u}\Lambda(c) \opp \Lambda(\xi)$$
Next there is a map $g := (g_r)$, the components $g_r : \Cone(f_r) \to \Cone(f_{r+1})$ are given by $g_r(c) := c$, $g_r(1_c):=1_c$ and $g_r(x):=0$ for $x\not\in \Lambda(c)$. The cone is 
$$\Cone(g_r) \simeq [\vnp{u_r}\Lambda(c_r) \opp \Lambda(\xi_r)]\opp [\vnp{u_{r+1}}\Lambda(c_{r+1}) \opp \Lambda(\xi_{r+1})]$$
where we include subscripts on the generators to emphasize their different graded degrees. There is a map $h=(h_r)$, where $h_r : \Cone(g_r)\to \a\Cone(g_{r+1})$ with $\a := \vnp{u_r}/\vnp{u_{r+1}} = t^{6r-4} q^{4r+1}/(t^{6(r+1)-4} q^{4(r+1)+1})=t^{-6}q^{-4}$ is independent of $r$, $h_r$ identifies the terms indicated below
$$\vnp{u_r}\Lambda(c_r) \xto{\sim} \a\vnp{u_{r+1}} \Lambda(c_{r+1})\conj{ and } \vnp{u_{r+1}}\Lambda(c_{r+1}) \xto{\sim} \a\vnp{u_{r+2}} \Lambda(c_{r+2})$$
and $h_r$ is zero on other terms. So that the cone $\Cone(h_r)$ can be identified with four terms
$$\Cone(h_r) \simeq \Lambda(\xi_r) \opp \Lambda(\xi_{r+1}) \opp \a \Lambda(\xi_{r+1}) \opp \a\Lambda(\xi_{r+2})$$
If $\beta := \vnp{\xi_r}/\vnp{\xi_{r+1}} = t^{-4}q^{-4}$ then there is a degree zero isomorphism $i_r : \Cone(h_r)\xto{\sim}\beta\Cone(h_{r+1})$ which identifies all of the terms below pairwise,
\[
i_r : \left[\begin{array}{c}
\Lambda(\xi_r)\\
\Lambda(\xi_{r+1})\\
\a \Lambda(\xi_{r+1})\\
\a\Lambda(\xi_{r+2}) 
  \end{array}\right]
\xto{\sim}
\beta \left[\begin{array}{c}
\Lambda(\xi_{r+1})\\
\Lambda(\xi_{r+2})\\
\a \Lambda(\xi_{r+2})\\
\a\Lambda(\xi_{r+3}) 
  \end{array}\right]
\]
To recap, we have a sequence of maps upon which we take cones. First $f : \srt\to \xM\srt$, $g : \Cone(f) \to \xL\Cone(f)$, $h : \Cone(g) \to \a\xL\Cone(g)$ and $i : \Cone(h) \to \beta\xL\Cone(h)$. Finally, $i$ is a degree zero isomorphism, $\Cone(i)\cong 0$ in $H^0(\SSeq)$. Here is the diagram,
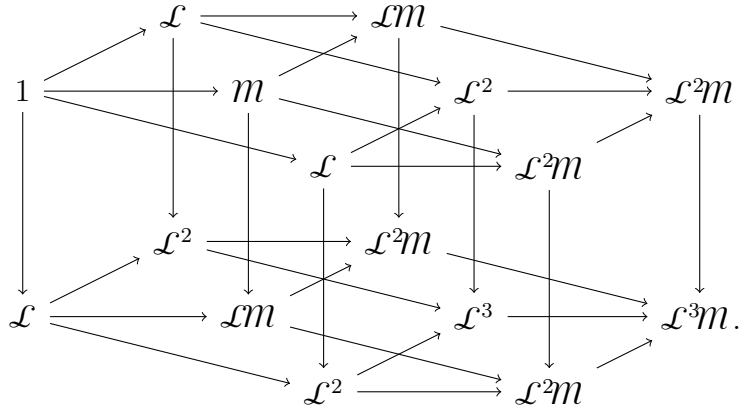
\begin{center}
\begin{tikzpicture}
\node (A) at (0,3) {$1$};
\node (B) at (3,3) {$\xM$};
\node (C) at (0,0) {$\xL$};
\node (D) at (3,0) {$\xL\!\xM$};
\draw[->] (A) to node {} (B);
\draw[->] (C) to node {} (D);
\draw[->] (A) to node [swap] {}  (C);
\draw[->] (B) to node {} (D);

\node (Ab) at (2,4) {$\xL$};
\node (Bb) at (5,4) {$\xL\!\xM$};
\node (Cb) at (2,1) {$\xL^2$};
\node (Db) at (5,1) {$\xL^2\!\xM$};
\draw[->] (Ab) to node {} (Bb);
\draw[->] (Cb) to node {} (Db);
\draw[->] (Ab) to node [swap] {}  (Cb);
\draw[->] (Bb) to node {} (Db);

\draw[->] (A) to node {} (Ab);
\draw[->] (B) to node {} (Bb);
\draw[->] (C) to node {} (Cb);
\draw[->] (D) to node {} (Db);

\node (A2) at (4,2) {$\xL$};
\node (B2) at (7,2) {$\xL^2\!\xM$};
\node (C2) at (4,-1) {$\xL^2$};
\node (D2) at (7,-1) {$\xL^2\!\xM$};
\draw[->] (A2) to node {} (B2);
\draw[->] (C2) to node {} (D2);
\draw[->] (A2) to node [swap] {}  (C2);
\draw[->] (B2) to node {} (D2);

\node (Ab2) at (6,3) {$\xL^2$};
\node (Bb2) at (9,3) {$\xL^2\!\xM$};
\node (Cb2) at (6,0) {$\xL^3$};
\node (Db2) at (9,0) {$\xL^3\!\xM$.};
\draw[->] (Ab2) to node {} (Bb2);
\draw[->] (Cb2) to node {} (Db2);
\draw[->] (Ab2) to node [swap] {}  (Cb2);
\draw[->] (Bb2) to node {} (Db2);

\draw[->] (A2) to node {} (Ab2);
\draw[->] (B2) to node {} (Bb2);
\draw[->] (C2) to node {} (Cb2);
\draw[->] (D2) to node {} (Db2);

\draw[->] (A) to node {} (A2);
\draw[->] (B) to node {} (B2);
\draw[->] (C) to node {} (C2);
\draw[->] (D) to node {} (D2);

\draw[->] (Ab) to node {} (Ab2);
\draw[->] (Bb) to node {} (Bb2);
\draw[->] (Cb) to node {} (Cb2);
\draw[->] (Db) to node {} (Db2);
\end{tikzpicture} 
\end{center}

  \end{example}

The diagram in Ex. \ref{exampletrefoilcomplex} is a kind of Koszul resolution. In \S\ref{recsec} we will show that for each knot $K$ there is a Koszul resolution of the sort pictured above which categorifies the Alexander A-polynomial $\theA_K$ from Thm. \ref{alexqholthm}.

\section{Holonomicity for $\fcfk(K)$}\label{recsec}

In this section, for each knot $K$ in $S^3$, we will construct a homological
$q$-holonomic relation in the sense of Def. \ref{homqholdef} for each sequence $S := (\fcfk(K))_{r\geq 0})$ 
The arguments here generalize
Ex. \ref{exampleunknotcomplex} and Ex. \ref{exampletrefoilcomplex}.

Recall Proposition \ref{propprop}, the decomposition is analogous to the observation that there is a bijection between the terms summed to produce the colored Alexander polynomials  from Def. \ref{coloredalexdef}. Allowing $r$ to vary shows that the ungraded chain complex remains constant up to changes in grading. (The grading transforms according to Prop. \ref{prop:grading}.) The maps $\ell_r$ introduced below make this idea more precise.

\begin{corollary}\label{ellrdef}
  There are isomorphisms of ungraded chain complexes
  $\ell_r : \fcfk(K) \to \ffcfk(K)$
  \end{corollary}
\begin{proof}
Set $\ell_r := \left(\sum_i r_i\right) + t$ in
\begin{center}
\begin{tikzpicture}[scale=10, node distance=1cm]
\node (A) {$\fcfk(K)$};
\node (B) [right=1.25cm of A] {$\left(\opp_i\Lambda(a^r_i)\right)\opp \theta^r\fcfk(U)$};
\node (C) [below=1cm of A] {$\ffcfk(K)$};
\node (D) [below=1cm of B] {$\left(\opp_i\Lambda(a^r_i)\right)\opp \theta^{r+1}\ffcfk(U)$};
\draw[->] (A) to node {$\sim$} (B);
\draw[->] (C) to node {$\sim$} (D);
\draw[->] (A) to node [swap] {$\ell_r$}  (C);
\draw[->] (B) to node {$(\sum_i r_i) + t$} (D);
\end{tikzpicture} 
\end{center}
The maps are determined by the following assignments, for  $r_i : \Lambda(a_{i,r}) \to \Lambda(a_{i,r+1})$ set $r_i(1_{i,r}) := 1_{i,r+1}$ and $r_i(a_{i,r}) := a_{i,r+1}$, and for $t : \theta^r\fcfk(U) \to \theta^{r+1}\ffcfk(U)$ set $t(u) := u$ and $t(\xi) := \xi$.
  \end{proof}

Let us now prove the reduced case (in analogy with Lemma \ref{knotlem}), the
head of the sequence in Prop. \ref{propprop} can be systematically reduced
with cones of the form $\Cone(\fD_i)$ ($\fD_i$ is introduced in the proof of Lemma \ref{reducedlemma} below). The graded Euler characteristic of these cones
agrees with the action of operators $D_X$ from \S\ref{Dopsec}. The main
result of this section, Thm. \ref{mainresultthm} is proven by reducing the
general case, the sequence containing the tail which is determined by Eq. \eqref{decompeqn}, to the
reduced case.

\begin{defn}\label{redseqdef}
A sequence $S := (S^r)_{r\geq 0} \in \SSeq$ will be called a {\em a reduced sequence} when each term
can be written as $S^r \cong \opp_{i\in I_r} \Lambda(a_{i,r})$ in $H^0(\SSeq)$, with $|I_r|<\infty$ and the maps identifying adjacent terms
\begin{equation}\label{lassigneq}
  \ell_r(a_{i,r}):= a_{i,r+1}\conj{ and } \ell_r(1_{i,r}):= 1_{i,r+1}
  \end{equation}
are (ungraded) isomorphisms $\ell_r : S^r \xto{\sim} S^{r+1}$ of chain complexes.
\end{defn}

\begin{prop}\label{equivprop}
For any such reduced sequence $S$, there are equivalence relations $\equiv$ on the index sets $I_r$. 
For $i,j\in I_r$,
  $$i \equiv j \conj{ when } \frac{\vnp{\ell_r(a_{i,r})}}{\vnp{a_{i,r}}} = \frac{\vnp{\ell_r(a_{j,r})}}{\vnp{a_{j,r}}}$$
\end{prop}

So for a reduced sequence $S$, we set $[I_r]:=I_r/\!\!\,\equiv$, $\vnp{S_r}:=\vnp{[I_r]}$ and $[\Lambda(a_{i,r})]:=\opp_{j\equiv i} \Lambda(a_{j,r})$ for the sum over corresponding congruence classes. It follows that each term $S^r$ in the sequence $S$ admits a factorization
  $$S^r \cong \sum_{i\in [I_r]} [\Lambda(a_{i,r})]$$
into congruence classes of summands. The maps $\ell_r$ respect this decomposition in the sense that $\ell_r([\Lambda(a_{i,r})]) = [\Lambda(\ell_r(a_{i,r}))]$.
There are maps $\ell_r : I_r \to I_{r+1}$ given by $\ell_r(i) := i$. 
\begin{defn}
A reduced sequence is {\em regular} when $i\equiv j$ if and only if $\ell_r(i) \equiv \ell_r(j)$. 
\end{defn}

\begin{lemma}\label{reducedlemma}
A regular reduced sequence $S$ is homologically $q$-holonomic in the sense of Def. \ref{homqholdef}.
\end{lemma}
\begin{proof}
  By induction on the number of congruence classes $\vnp{S}$. If $\vnp{S} = 0$ then the sequence $S=\ul{0}$ is the constant sequence. Suppose that $\vnp{S}=n$ and pick $i\in [I_0]$ then the maps $\ell_r$ restrict to degree zero isomorphisms $f_{i,r} := \ell_r|_{[\Lambda(a_{i,r})]}$
  $$f_{i,r} : [\Lambda(a_{i,r})]\to \a_i [\Lambda(\ell_r(a_{i,r}))] \conj{ where } \a_i := \frac{\vnp{\ell_r(a_{i,r})}}{\vnp{a_{i,r}}}$$
So there are degree zero chain maps $\fD_{i,r} : S^r \to \a_i\xL S^{r}$
\[
\fD_{i,r}(x) := \left\{ \begin{array}{ll} f_{i,r}(x) & x\in [\Lambda(a_{i,r})]\\
  0 & x\not\in [\Lambda(a_{i,r})]
\end{array}\right.
\]
There is a map $\fD_i := (\fD_{i,r})_{r\geq 0}$ formed by combining the maps $\fD_{i,r}$. 
$$\fD_i := \oplus_{j\equiv i} \fD_{j,r}$$
The cone $S_{new} := \Cone(\fD_{i})$ is homotopy equivalent to a sequence which excludes the $i$th component of $S$ and doubles the others
\begin{equation}\label{fancypizzaeq}
  \Cone(\fD_{i,r}) \cong S^r_{\backslash i} \opp \a_i S^r_{\backslash i} \conj{ where we set } S^r_{\backslash i} := \opp_{j\in [I_r]\backslash \{ i\} } [\Lambda(a_{j,r})].
  \end{equation}

The proof is concluded by showing that the sequence $\Cone(\fD_i)$ is a regular reduced sequence with $\vnp{\Cone(\fD_{i})} = \vnp{S} - 1$. This will occur in three steps. First we prove that the sequence is reduced, second that the sequence is regular and third we show that the number of congruence classes has decreased by one. The point is that the set of congruence classes of $S^r_{\backslash i}$ and $\alpha_i S^r_{\backslash i}$ are the same because $\alpha_i$ is just a degree shift.
 
\begin{enumerate}[label=(\roman*),labelwidth=!, labelindent=0pt, wide]
\item {\it The sequence $\Cone(\fD_i)$ is reduced:} Eqn. \eqref{fancypizzaeq} shows that each term 
$$\Cone(\fD_i)_r \cong \bigoplus_{i\not\equiv j} \left(\Lambda(a_{j,i}) \opp \a_i\Lambda(a_{j,i}) \right) = \bigoplus_{j\in J_r} \Lambda(b_{j,i})$$
where $J_r := A_r\sqcup B_r$ where $A_r := I_r\backslash [i]$ and  $B_r:=I_r\backslash [i]$  and an element $j\in I\backslash [i]$ corresponds to either $\Lambda(b_{j,r}):=\Lambda(a_{j,r})$ when $j\in A_r$ and $\Lambda(b_{j,r}):=\a_i\Lambda(a_{j,r})$ when $j\in B_r$. 

The maps $m_r : \opp_{j\in J_r} \Lambda(b_{i,r})\to \opp_{j\in J_{r+1}} \Lambda(b_{i,r+1})$ determined by the assignments in Eqn. \eqref{lassigneq} above determine isomorphisms between adjacent terms in the new sequence because regularity of the sequence $S$  shows $\ell_r$ respects the equivalence relation $\equiv$ and induces isomorphisms between components. The maps $m_r = \ell_r|_{S^r/[\Lambda(a_{i,r})]} + \a_i\ell_r|_{S^r/[\Lambda(a_{i,r})]}$ are a block diagonal matrix with components given by the restriction of $\ell_r$ to the complement of the $i$th component and an $\a_i$-shifted copy of this restriction. 

\item {\it The sequence $\Cone(\fD_i)$ is regular:} Since $(\a_i \vnp{\ell_r(a_{j,r})}/(\a_i \vnp{a_{j,r}} )= \vnp{\ell_r(a_{j,r})}/\vnp{a_{j,r}}$ the fold map $\varphi : J_r \to I_r\backslash [i]$ is 2-to-1 and induces a bijection $J_r/\!\!\,\equiv\,\, \xto{\sim} (I_r\backslash [i])/\!\!\,\equiv$. Since $S^r$ is regular, $j,k \in I_r\backslash [i]$ implies $j \equiv k \Leftrightarrow \ell_r(j)\equiv \ell_r(k)$. Since $m_r = \ell_r$ on equivalence classes under the identification $\varphi$, $\Cone(\fD_r)$ is regular as well.

\item {\it The number of congruence classes has decreased:} Observe that $\vnp{\Cone(\fD_r)}= \vnp{J_r/\!\!\,\equiv} =_{\varphi} |(I_r/\!\equiv)\backslash [i]|= \vnp{I_r/\!\!\,\equiv}-1$.

\end{enumerate}
\end{proof}

\begin{thm}\label{mainresultthm} 
If $K$ is a knot then the sequence $S^r(K) := (\fcfk(K))_{r\geq 0} \in \SSeq$ is homologically $q$-holonomic.
\end{thm}
\begin{proof}
Recall from Prop. \ref{propprop} that 
$$\fcfk(K) \cong \Hd \opp \Tl$$
where $\Hd := \left(\Lambda(a_{1,r})\opp\Lambda(a_{2,r})\opp\cdots\opp\Lambda(a_{n,r})\right)$ and $\Tl := \theta^r\fcfk(U)$.

There is a map $f_r : \fcfk(K)\to \vnp{u}\fcfk(K)$ which vanishes on $\Hd$,
$f|_{\Hd}=0$, and on the tail is defined to be $f(u^n\xi^k) := u^{n-1}\xi^k$. As in Ex. \ref{exampletrefoilcomplex}, 
$$\Cone(f_r) \simeq \Hd \opp \vnp{u} \Hd \opp \Lambda(\xi).$$
 The maps $f_r$ combine to form a map $f : \fcfk(K) \to\xM\fcfk(K)$. We claim that the sequence of cones $\Cone(f) \in \SSeq$ is regular and reduced.

Set $I_r := \{1,\ldots,n,\vnp{u}1, \ldots, \vnp{u}n, \xi\}$. It follows from Prop. \ref{propprop} that
$$\Cone(f_r)  \cong \Hd\opp \vnp{u}\Hd\opp \Lambda(\xi) = \sum_{i\in I_r} \Lambda(c_{i,r})$$
where 
$$\Lambda(c_{i,r}) :=\left\{\begin{array}{ll}
\Lambda(a_{i,r}) & i\in\{1,\ldots,n\}\\
\vnp{u}\Lambda(a_{i,r}) & i \in\{ \vnp{u}1,\ldots,\vnp{u}n\}\\
\Lambda(\xi) & i\in\{\xi\}
\end{array}\right.$$
this is the same complex for each $r\geq 1$, so the maps $\ell_r : \Cone(f_r) \to \Cone(f_{r+1})$ given by the assignments Eqn. \eqref{lassigneq} are isomorphisms. It follows that the sequence $\Cone(f)$ is reduced. 

The sequence $\Cone(f)$ is regular because, for any $i\in I_r$, the bidegrees of $\vnp{1_{i,r}}$ are fixed affine linear functions of $r$, which implies that the quotients $\vnp{\ell_r(1_{i,r})}/\vnp{1_{i,r}}$ are constant functions of $r$.

The result now follows from Lemma \ref{reducedlemma}.
\end{proof}

\begin{remark}
The theorem shows that holonomicity corresponds to a Koszul-type complex which is determined by the immersed curve representing the module $\widehat{CFD}(K)$.
  \end{remark}

\section{Notation}\label{app2sec}

\noindent
\begin{multicols}{2}
\begin{list}{}{
  \renewcommand{\makelabel}[1]{#1\hfil}
}
\item[$\equiv$] Prop. \ref{equivprop}
\item[$\circ$] Notation \ref{circlenotation}
\item[$\vnp{x}_t$] homological degree 
\item[$\vnp{x}_q$] graded degree 
\item[$\vnp{x}_{t,q}$] $\vnp{x}_{t,q}=(\vnp{x}_t, \vnp{x}_q)$
\item[$\vnp{x}$] multiplicative bidegree, if $\vnp{x}_{t,q} = (n,m)$ then $\vnp{x} = t^nq^m$ 
\item[$1_{i,r}$] Prop. \ref{propprop}
\item[$3_1$] right-handed trefoil
\item[$\theta$] Prop. \ref{propprop}
\item[$\Delta_n$] Ex. \ref{example:dilate}
\item[$\alex_K(q)$] Alexander's balanced Laurent polynomial 
\item[$\Delta^1_K(q)$] positive form of Alexander polynomial
\item[$\unsralex_K(q)$] unreduced $r$-colored polynomial, Def. \ref{coloredalexdef}
\item[$\unsralex_U(q)$] $r$-colored unknot, Eqn. \eqref{unknoteq}
\item[$\sralex_K(q)$] reduced $r$-colored polynomial, Eqn. \eqref{redeq}
\item[$a_{i,r}$] Prop. \ref{propprop}
\item[$\weyl$] Weyl algebra \S\ref{weylalgsec}
\item[$\xA^+$] Eqn. \eqref{texeq}
\item[$\tA_K$] $A$-polynomial, Lem. \ref{knotlem}
\item[$\bar{\tA}_K$] unreduced $A$-polynomial, Thm. \ref{alexqholthm}
\item[$\eA_\ZZ$] category of graded object, \S\ref{app1sec}
\item[$\bgen$] bottom generator,  Not. \ref{notation:vertsimp}
\item[$\dgen$] distinguished generator, Not. \ref{notation:vertsimp}
\item[$\bar{C}$] Ex. \ref{constseqex}
\item[$\hC(K)$,$C^-(K)$] Eqn. \eqref{eq:hatc}
\item[$CF(\ga,\eta)$] Thm. \ref{thm:ICpairing}
\item[$\Ch(\eA)$] category of chain complexes in $\eA$, Def. \ref{chdef}
\item[$\Do_X$] difference operator, \S\ref{Dopsec}
\item[$\fD$] Lem. \ref{reducedlemma}
\item[$F^*\eA$] category of filtered objects, \S\ref{app1sec}
\item[$\ga_K, \og_K$]  \S\ref{ICs}
\item[$\hog_K$] Cor. \ref{cor:structurecor}
\item[$g(K)$] Seifert genus Rem. \ref{rem:kinv}
\item[$\og_0$] Rmk. \ref{rem:strprop} (1) 
\item[$H^0$] homotopy category, \S\ref{app1sec}
\item[$\Hd$] Def. \ref{onecoloreddef}
\item[$\hfk(K)$] $H_*(\cfk(K))$, knot Floer homology (balanced grading)
\item[$K_{r,s}$] Def. \ref{def:cable}
\item[$\ell_r$] Cor. \ref{ellrdef}
\item[$\xL$] Prop. \ref{weylprop}
\item[$\xM$] Prop. \ref{weylprop}
\item[$\nu K$] regular neighborhood of knot
\item[$q^m$] monomial or graded shift functor $q^m = \cdot\{-m\}$ 
\item[$\Seq$] Def. \ref{seq}
\item[$\Seq(T)$] Def. \ref{seq}
\item[$\SSeq$] $\SSeq = \SSeq(F^*\Ch_\ZZ)$, Eqn. \eqref{sseqdef}
\item[$\SSeq(\eA)$] sequences in the dg category $\eA$
\item[$\tau(K)$] Tau invariant Rem. \ref{rem:kinv}
\item[$t^n$] monomial or homological shift functor $t^n = \cdot[-n]$
\item[$t^nq^m$] monomial or shift functor
\item[$T_K, \bar{T}_K$] \S\ref{ICs}
\item[$\Tl$] Def. \ref{onecoloreddef}
\item[$\cfd$] Eqn. \eqref{eq:pairing}
\item[$\gscfk(K)$] $S^r$-colored knot Floer associated graded chain complex
\item[$\scfk(K)$] $S^r$-colored knot Floer filtered graded chain complex
\item[$\shfk(K)$] $H_*(\scfk(K))$
\item[$U$] the unknot
\item[$\ZZg$] $\ZZg := \{0,1,2,\ldots\}\subset \ZZ$
\end{list}
\end{multicols}

\bibliography{stablehfk}{}

\newcommand{\etalchar}[1]{$^{#1}$}
\providecommand{\bysame}{\leavevmode\hbox to3em{\hrulefill}\thinspace}
\providecommand{\MR}{\relax\ifhmode\unskip\space\fi MR }
\providecommand{\MRhref}[2]{%
  \href{http://www.ams.org/mathscinet-getitem?mr=#1}{#2}
}
\providecommand{\href}[2]{#2}
\begin{thebibliography}{BPLHW22}

\bibitem[Abo08]{Abouzaid}
Mohammed Abouzaid, \emph{On the {F}ukaya categories of higher genus surfaces},
  Adv. Math. \textbf{217} (2008), no.~3, 1192--1235. \MR{2383898}

\bibitem[BK90]{BK}
A.~I. Bondal and M.~M. Kapranov, \emph{Framed triangulated categories}, Mat.
  Sb. \textbf{181} (1990), no.~5, 669--683. \MR{1055981}

\bibitem[BPLHW22]{BPLW}
Anna Beliakova, Krzysztof Putyra, Robert Louis-Hadrien, and Emmanuel Wagner,
  \emph{A proof of {D}unfield-{G}ukov-{R}asmussen conjecture}, arXiv:2210.00878
  (2022).

\bibitem[Bri05]{Brion}
Michel Brion, \emph{Lectures on the geometry of flag varieties}, Topics in
  cohomological studies of algebraic varieties, Trends Math., Birkh\"{a}user,
  Basel, 2005, pp.~33--85. \MR{2143072}

\bibitem[Cau17]{cautisremarks}
Sabin Cautis, \emph{Remarks on coloured triply graded link invariants}, Algebr.
  Geom. Topol. \textbf{17} (2017), no.~6, 3811--3836. \MR{3709661}

\bibitem[CCG{\etalchar{+}}94]{CCGLS}
D.~Cooper, M.~Culler, H.~Gillet, D.~D. Long, and P.~B. Shalen, \emph{Plane
  curves associated to character varieties of {$3$}-manifolds}, Invent. Math.
  \textbf{118} (1994), no.~1, 47--84. \MR{1288467}

\bibitem[Che22]{Chen22}
Daren Chen, \emph{Twistings of the {A}lexander polynomial},
  https://arxiv.org/pdf/2202.12489 (2022).

\bibitem[DGR06]{DGR}
Nathan~M. Dunfield, Sergei Gukov, and Jacob Rasmussen, \emph{The
  superpolynomial for knot homologies}, Experiment. Math. \textbf{15} (2006),
  no.~2, 129--159. \MR{2253002}

\bibitem[FG00]{FGSKM}
Charles Frohman and R\u{a}zvan Gelca, \emph{Skein modules and the
  noncommutative torus}, Trans. Amer. Math. Soc. \textbf{352} (2000), no.~10,
  4877--4888. \MR{1675190}

\bibitem[FGL02]{FGL}
Charles Frohman, R\u{a}zvan Gelca, and Walter Lofaro, \emph{The {A}-polynomial
  from the noncommutative viewpoint}, Trans. Amer. Math. Soc. \textbf{354}
  (2002), no.~2, 735--747. \MR{1862565}

\bibitem[FGSk13]{FGS}
Hiroyuki Fuji, Sergei Gukov, and Piotr Su\l~kowski,
  \emph{Super-{$A$}-polynomial for knots and {BPS} states}, Nuclear Phys. B
  \textbf{867} (2013), no.~2, 506--546. \MR{2992793}

\bibitem[Gar04]{G2}
Stavros Garoufalidis, \emph{On the characteristic and deformation varieties of
  a knot}, Proceedings of the {C}asson {F}est, Geom. Topol. Monogr., vol.~7,
  Geom. Topol. Publ., Coventry, 2004, pp.~291--309. \MR{2172488}

\bibitem[Gar08]{G1}
\bysame, \emph{Difference and differential equations for the colored {J}ones
  function}, J. Knot Theory Ramifications \textbf{17} (2008), no.~4, 495--510.
  \MR{2414452}

\bibitem[GL05]{GL}
Stavros Garoufalidis and Thang T.~Q. L\^{e}, \emph{The colored {J}ones function
  is {$q$}-holonomic}, Geom. Topol. \textbf{9} (2005), 1253--1293. \MR{2174266}

\bibitem[GL16]{GLsurvey}
\bysame, \emph{A survey of {$q$}-holonomic functions}, Enseign. Math.
  \textbf{62} (2016), no.~3-4, 501--525. \MR{3692896}

\bibitem[GLL18]{GLL}
Stavros Garoufalidis, Aaron~D. Lauda, and Thang T.~Q. L\^{e}, \emph{The colored
  {HOMFLYPT} function is {$q$}-holonomic}, Duke Math. J. \textbf{167} (2018),
  no.~3, 397--447. \MR{3761103}

\bibitem[Guk05]{Gukov}
Sergei Gukov, \emph{Three-dimensional quantum gravity, {C}hern-{S}imons theory,
  and the {A}-polynomial}, Comm. Math. Phys. \textbf{255} (2005), no.~3,
  577--627. \MR{2134725}

\bibitem[Han23a]{Han23a}
Jonathan Hanselman, \emph{Heegaard {F}loer homology and cosmetic surgeries in
  {$S^3$}}, J. Eur. Math. Soc. (JEMS) \textbf{25} (2023), no.~5, 1627--1669.
  \MR{4592856}

\bibitem[Han23b]{Hanselman23}
\bysame, \emph{Knot {F}loer homology as immersed curves},
  https://arxiv.org/abs/2305.16271 (2023).

\bibitem[HKK17]{HKK}
F.~Haiden, L.~Katzarkov, and M.~Kontsevich, \emph{Flat surfaces and stability
  structures}, Publ. Math. Inst. Hautes \'{E}tudes Sci. \textbf{126} (2017),
  247--318. \MR{3735868}

\bibitem[Hog18]{hogancamp}
Matthew Hogancamp, \emph{Categorified {Y}oung symmetrizers and stable homology
  of torus links}, Geom. Topol. \textbf{22} (2018), no.~5, 2943--3002.
  \MR{3811775}

\bibitem[Hom20]{Hom20}
Jennifer Hom, \emph{Lecture notes on {H}eegaard {F}loer homology},
  https://arxiv.org/abs/2008.01836 (2020).

\bibitem[HRW22]{HRW2}
Jonathan Hanselman, Jacob Rasmussen, and Liam Watson, \emph{Heegaard {F}loer
  homology for manifolds with torus boundary: properties and examples}, Proc.
  Lond. Math. Soc. (3) \textbf{125} (2022), no.~4, 879--967. \MR{4500201}

\bibitem[HRW24]{HRW1}
\bysame, \emph{Bordered {F}loer homology for manifolds with torus boundary via
  immersed curves}, J. Amer. Math. Soc. \textbf{37} (2024), no.~2, 391--498.
  \MR{4695506}

\bibitem[HW23a]{HW23}
Jonathan Hanselman and Liam Watson, \emph{Cabling in terms of immersed curves},
  Geom. Topol. \textbf{27} (2023), no.~3, 925--952. \MR{4599309}

\bibitem[HW23b]{HW23b}
\bysame, \emph{A calculus for bordered {F}loer homology}, Geom. Topol.
  \textbf{27} (2023), no.~3, 823--924. \MR{4599308}

\bibitem[KMS09]{KMS}
Mikhail Khovanov, Volodymyr Mazorchuk, and Catharina Stroppel, \emph{A brief
  review of abelian categorifications}, Theory Appl. Categ. \textbf{22} (2009),
  No. 19, 479--508. \MR{2559652}

\bibitem[KR08]{KR1}
Mikhail Khovanov and Lev Rozansky, \emph{Matrix factorizations and link
  homology}, Fund. Math. \textbf{199} (2008), no.~1, 1--91. \MR{2391017}

\bibitem[L\^06]{Le}
Thang T.~Q. L\^{e}, \emph{The colored {J}ones polynomial and the
  {$A$}-polynomial of knots}, Adv. Math. \textbf{207} (2006), no.~2, 782--804.
  \MR{2271986}

\bibitem[LC18]{Lambertcole}
Peter Lambert-Cole, \emph{Twisting, mutation and knot {F}loer homology},
  Quantum Topol. \textbf{9} (2018), no.~4, 749--774. \MR{3874002}

\bibitem[Lic97]{Lickorish}
W.~B.~Raymond Lickorish, \emph{An introduction to knot theory}, Graduate Texts
  in Mathematics, vol. 175, Springer-Verlag, New York, 1997. \MR{1472978}

\bibitem[LOT18]{LOT18b}
Robert Lipshitz, Peter~S. Ozsvath, and Dylan~P. Thurston, \emph{Bordered
  {H}eegaard {F}loer homology}, Mem. Amer. Math. Soc. \textbf{254} (2018),
  no.~1216, viii+279. \MR{3827056}

\bibitem[OS04a]{OZ}
Peter Ozsv\'{a}th and Zolt\'{a}n Szab\'{o}, \emph{Holomorphic disks and knot
  invariants}, Adv. Math. \textbf{186} (2004), no.~1, 58--116. \MR{2065507}

\bibitem[OS04b]{OS03prop}
Peter Ozsv\'ath and Zolt\'an Szab\'o, \emph{Holomorphic disks and
  three-manifold invariants: properties and applications}, Ann. of Math. (2)
  \textbf{159} (2004), no.~3, 1159--1245. \MR{2113020}

\bibitem[OS04c]{OS03}
\bysame, \emph{Holomorphic disks and topological invariants for closed
  three-manifolds}, Ann. of Math. (2) \textbf{159} (2004), no.~3, 1027--1158.
  \MR{2113019}

\bibitem[Ras03a]{R}
Jacob Rasmussen, \emph{Floer homology and knot complements}, PhD thesis,
  Harvard University (2003).

\bibitem[Ras03b]{Ras03}
Jacob~Andrew Rasmussen, \emph{Floer homology and knot complements}, ProQuest
  LLC, Ann Arbor, MI, 2003, Thesis (Ph.D.)--Harvard University. \MR{2704683}

\bibitem[Ras15]{RKR}
Jacob Rasmussen, \emph{Some differentials on {K}hovanov-{R}ozansky homology},
  Geom. Topol. \textbf{19} (2015), no.~6, 3031--3104. \MR{3447099}

\bibitem[Roz14]{roz}
Lev Rozansky, \emph{An infinite torus braid yields a categorified
  {J}ones-{W}enzl projector}, Fund. Math. \textbf{225} (2014), no.~1, 305--326.
  \MR{3205575}

\bibitem[Sab93]{Sabbah}
Claude Sabbah, \emph{Syst\`emes holonomes d'\'{e}quations aux
  {$q$}-diff\'{e}rences}, {$D$}-modules and microlocal geometry ({L}isbon,
  1990), de Gruyter, Berlin, 1993, pp.~125--147. \MR{1206016}

\bibitem[Sar15]{Sartori}
Antonio Sartori, \emph{The {A}lexander polynomial as quantum invariant of
  links}, Ark. Mat. \textbf{53} (2015), no.~1, 177--202. \MR{3319619}

\bibitem[Vir06]{Viro}
O.~Ya. Viro, \emph{Quantum relatives of the {A}lexander polynomial}, Algebra i
  Analiz \textbf{18} (2006), no.~3, 63--157. \MR{2255851}

\bibitem[Web17]{Webster}
Ben Webster, \emph{Knot invariants and higher representation theory}, Mem.
  Amer. Math. Soc. \textbf{250} (2017), no.~1191, v+141. \MR{3709726}

\bibitem[Wei13]{Kbook}
Charles~A. Weibel, \emph{The {$K$}-book}, Graduate Studies in Mathematics, vol.
  145, American Mathematical Society, Providence, RI, 2013, An introduction to
  algebraic $K$-theory. \MR{3076731}

\end{thebibliography}
\bibliographystyle{amsalpha}

\end{document}